\RequirePackage{ifthen}
\newboolean{springer}
\setboolean{springer}{false}

\ifthenelse {\boolean{springer}}
{
\documentclass{svjour3}
\smartqed
\usepackage[margin=1.3in]{geometry}
} {
\documentclass[11pt]{article}
\usepackage[margin=1in]{geometry}
}

\usepackage{CJKutf8}
\usepackage{amsmath}
\usepackage{graphicx}
\usepackage{psfrag}
\usepackage{amssymb}
\usepackage{amsmath}
\usepackage{verbatim,booktabs}
\usepackage{subfigure,epsfig,url}
\usepackage{float}
\usepackage{mathrsfs}
\usepackage{enumitem} 
\usepackage{caption}

\usepackage[usenames,dvipsnames]{pstricks}
\usepackage{epsfig}
\usepackage{pst-grad} 
\usepackage{pst-plot} 
\usepackage{xspace}
\usepackage{bbm}
\usepackage{comment}
\usepackage{tikz}
\usepackage{accents}
\usepackage{multirow}
\usepackage{array}
\usepackage{longtable}
\usepackage{pdflscape}
\usepackage[ruled,vlined,linesnumbered,noresetcount]{algorithm2e}
\usepackage{hyperref}

\makeatletter
\let\cl@chapter\undefined
\makeatletter
\usepackage{cleveref}

\ifthenelse {\boolean{springer}}
{

\newenvironment{prf}
{\begin{trivlist} \item[] {\em Proof }}
{$\hfill\qed$ \end{trivlist}}

\newenvironment{prfc}[1][]
{\begin{trivlist} \item[] {\em Proof #1 }}
{$\hfill\qed$ \end{trivlist}}

\newcounter{claim} 
\renewenvironment{claim}[1][]
{\refstepcounter{claim} \begin{trivlist} \item[] {\bf Claim~\theclaim}\space#1 \itshape}
{\end{trivlist}}

\newenvironment{cpf}
{\begin{trivlist} \item[] {\em Proof of claim }}
{$\hfill\diamond$ \end{trivlist}}

\newenvironment{spf}
{\begin{trivlist} \item[] {\em Proof of step }}
{$\hfill\triangleleft$ \end{trivlist}}

\journalname{Journal name}

} {

\usepackage{amsthm}
\newtheorem{theorem}{Theorem}
\newtheorem{proposition}{Proposition}
\newtheorem{lemma}{Lemma}
\newtheorem{corollary}{Corollary}

\newtheorem{definition}{Definition}

\newenvironment{prf}[1][]
{\begin{proof}}
{\end{proof}}

}

\def\ie{{i.e.,} }
\def\eg{{e.g., }}
\newcommand{\retal} {{et al.\/} }

\newcommand{\orgdiv}[1]{#1}%
\newcommand{\orgname}[1]{#1}%
\newcommand{\orgaddress}[1]{#1}%
\newcommand{\postcode}[1]{#1}%
\newcommand{\city}[1]{#1}%
\newcommand{\country}[1]{#1}%

\newcommand{\fnm}[1]{#1}%
\newcommand{\sur}[1]{#1}%

\newcolumntype{L}{>{$}l<{$}}
\newcolumntype{C}{>{$}c<{$}}
\newcolumntype{R}{>{$}r<{$}}

\newcommand{\bR}{\mathbb{R}}

\newcommand{\cB}{{\mathcal B}}

\newcommand{\cC}{{\mathcal C}}
\newcommand{\cK}{{\mathcal K}}
\newcommand{\cD}{{\mathcal{D}}}

\newcommand{\sF}{{\mathsf F}}
\newcommand{\cP}{{\mathcal P}}
\newcommand{\cR}{{\mathcal R}}

\newcommand{\bS}{{\mathbb S}}

\newcommand{\cS}{{\mathcal S}}

\newcommand{\cX}{{\mathcal X}}

\newcommand{\norm}[1]{{\lVert#1\rVert}}

\renewcommand{\t}[1]{{#1}^\top}
\newcommand{\relx}[1]{\tilde{#1}}

\SetKw{Break}{break}

\DeclareMathOperator{\epi}{epi}
\DeclareMathOperator{\hyp}{hypo}
\DeclareMathOperator{\gra}{gra}
\DeclareMathOperator{\conv}{conv}

\DeclareMathOperator{\bd}{bd}
\DeclareMathOperator{\ext}{ext}

\DeclareMathOperator{\argmax}{argmax}

\newcommand{\maxcut}{\textsc{max cut}\xspace}
\newcommand{\pbm}{\textsc{pseudo Boolean maximization}\xspace}
\newcommand{\bdopt}{\textsc{Bayesian D-optimal design}\xspace}
\newcommand{\dopt}{\textsc{D-optimal design}\xspace}

\DeclareMathOperator{\inter}{int}
\DeclareMathOperator{\relint}{relint}
\DeclareMathOperator{\relbd}{relbd}
\DeclareMathOperator{\conve}{env}

\DeclareMathOperator{\ldet}{ldet}
\newcommand{\cha}[1]{ \mathsf{supp}(#1)}

\newcommand{\ep}[1]{ \ensuremath{EPM_{#1}} }
\newcommand{\ee}[1]{\ensuremath{EE_{#1} }}

\newcommand{\leo}[1]{{\color{red}#1}}

\Crefname{chapter}{Chap.}{Chaps.}
\Crefname{section}{Sect.}{Sects.}
\Crefname{proposition}{Prop.}{Props.}
\Crefname{theorem}{Thm.}{Thms.}
\Crefname{definition}{Defn.}{Defns.}
\Crefname{corollary}{Cor.}{Cors.}
\Crefname{figure}{Fig.}{Figs.}
\Crefname{observation}{Obs}{Obss.}

\SetKwComment{Comment}{$\triangleright$\ }{}
\SetCommentSty{}
\SetKw{Continue}{continue}
\SetKw{Break}{break}

\allowdisplaybreaks

\newcommand{\bi}{\begin{list}{$\bullet$}{\setlength{\parsep}{0pt}\setlength{\itemsep}{0pt}}}

\title{Submodular maximization and its generalization through an intersection cut lens}

\ifthenelse {\boolean{springer}}
{
\titlerunning{Intersection Cuts for Submodular Maximization and Generalizations}
\authorrunning{Xu \retal}

\author{\fnm{Liding} \sur{Xu} \and  \fnm{Leo} \sur{Liberti} }

\institute{
  \fnm{Liding} \sur{Xu},  \fnm{Leo} \sur{Liberti}\at
  \orgdiv{LIX CNRS}, \orgname{\'{E}cole Polytechnique, Institut Polytechnique de Paris}, \orgaddress{\city{Palaiseau}, \postcode{91128}, \country{France}} \\
 \email{liding.xu@polytechnique.edu, liberti@lix.polytechnique.fr}
}

}
{
\author{
\fnm{Liding} \sur{Xu} \thanks{ \orgdiv{LIX CNRS}, \orgname{\'Ecole Polytechnique, Institut Polytechnique de Paris}, \orgaddress{\city{Palaiseau}, \postcode{91128}, \country{France}}.
    		 E-mail: {\tt liding.xu@polytechnique.edu,  liberti@lix.polytechnique.fr}}
    		 \and
\fnm{Leo} \sur{Liberti} \footnotemark[1]
}
}

\date{\today}

\begin{document}

\maketitle

\date{}

\begin{abstract}
We study a mixed-integer set $\cS:=\{(x,t) \in \{0,1\}^n \times \bR: f(x) \ge t\}$ arising in the submodular maximization problem, where $f$ is a submodular function defined over $\{0,1\}^n$. We use intersection cuts to tighten a polyhedral outer approximation of $\cS$. We construct  a continuous extension  $\sF$  of  $f$, which is convex and  defined over the entire space $\bR^n$. We show that the epigraph  $\epi(\sF)$ of  $\sF$  is an $\cS$-free set, and characterize maximal $\cS$-free sets including $\epi(\sF)$.  We propose a hybrid discrete Newton algorithm to compute an intersection cut efficiently and exactly. Our results are generalized to  the hypograph or the superlevel set of a submodular-supermodular function, which is a model for discrete nonconvexity. A consequence of these results is intersection cuts for Boolean multilinear constraints. We evaluate our techniques on  max cut, pseudo Boolean  maximization, and Bayesian D-optimal design problems within a MIP solver.

\ifthenelse {\boolean{springer}}
{
} {}
\end{abstract}

\ifthenelse {\boolean{springer}}
{}{
}

\section{Introduction}
In this paper, we consider the submodular maximization problem:
\begin{equation}
\label{eq.milp}
    \max_{t \in \bR} \, t \quad \textup{s.t.}  \quad f(x) \ge t, \quad x  \in \{0,1\}^n \cap \cX.   
\end{equation}
where $f:\{0,1\}^n \to \bR$ is a submodular function and $\cX \subseteq \bR^n$ is  a set describing additional constraints.  We study valid inequalities for the mixed-integer set $\hyp_{\{0,1\}^n}(f):=\{(x,t) \in \{0,1\}^n \times \bR: f(x) \ge t\}$, which is the hypograph of $f$ over the Boolean hypercube $\{0,1\}^n$.

The maximization of arbitrary submodular functions (i.e., Eq.~\eqref{eq.milp}) can be reduced to a Mixed-Integer Linear Program (MILP) with exponentially many linear inequalities \cite{nemhauser1978analysis}.   No polynomial-time algorithm is yet known to separate these inequalities. The Benders-like exact approach based on a branch-and-cut algorithm proposed in  \cite{coniglio2022submodular} provides  global dual bounds for primal solutions, and achieves a finite convergence rate.

 Many submodular maximization problems  (\eg max cut with positive edge weights \cite{schrijver2003combinatorial}, D-optimal design \cite{sagnol2015computing}, and utility maximization \cite{ahmed2011maximizing}) have natural MILP or mixed-integer nonlinear programming (MINLP) formulations, which can be solved using general-purpose global optimization solvers. The algorithm underlying these solvers is typically a branch-and-cut algorithm, which uses polyhedral outer approximations to construct LP relaxations \cite{bestuzheva2021scip,bestuzheva2023global,tawarmalani2005polyhedral}. For submodular maximization problems with  convex MINLP formulations,  a state-of-art algorithm also  uses  polyhedral outer approximations \cite{coey2020outer}.
 
Intersection cuts can be used to strengthen polyhedral outer approximations of a nonconvex set $\cS$ that is considered hard to optimize over. The construction of intersection cuts \cite{conforti2014} requires two key ingredients: a corner polyhedron relaxation of $\cS$, and an $\cS$-free set, which is a convex set that does not contain any interior point of $\cS$. (Inclusion-wise) maximal $\cS$-free sets generate strong intersection cuts  not dominated by other intersection cuts.

Intersection cuts were initially devised in the continuous setting (the papers \cite{tuy64,thy1985}, cited in \cite[Ch.~III]{horsttuy}, appeared before the classic paper \cite{balas1971intersection}), where they could approximate the hypograph $\cS$ of a convex function over a polytope.   There is a  unique  maximal  $\cS$-free set: the epigraph  of that convex function. Later, intersection cuts were used in the discrete setting \cite{balas1971intersection}, where $\cS$ is  a lattice. Several more families of lattice-free sets ({\eg}~splits, triangles, and spheres \cite{conforti2014,liberti2008spherical}) were described later.

The submodular maximization problem plays an intermediate role between these settings. On the one hand, the submodular function $f$ is defined over the Boolean hypercube $\{0,1\}^n$. Therefore,  the graph of $f$  projected on $\bR^n$ is  a subset of a lattice. On the other hand, as a discrete  analogue to convex functions, $f$ has a convex (thus continuous) extension over the hypercube $[0,1]^n$, namely the  Lovász extension \cite{lovasz1983submodular}. We can extend the Lovász extension to  a convex function, which we call $\sF$, over the entire $n$-dimensional Euclidean space $\bR^n$. This (continuous) function $\sF$ inherits a rich combinatorial structure from $f$.

The difference of two submodular functions (call them $f_1,f_2$) is a \textit{submodular-supermodular} (SS) function. SS functions generalize submodular functions, which are also discrete  analogues of  difference-of-convex (DC) functions. The epigraphs/superlevel sets of DC functions can represent various nonconvex sets, \eg quadratic sets \cite{munoz2022maximal} and signomial-term sets \cite{xusignomial}. This representation facilitates the derivation of intersection cuts \cite{serrano2019}. In fact,  SS functions may represent some discrete nonconvex functions arising in combinatorial optimization. For example, we will show that any Boolean multilinear function is an SS function.

Let $\cS=\hyp_{\{0,1\}^n}(f)$.  In this paper, we use convex extensions in order to construct some $\cS$-free (namely, hypograph-free) sets with Boolean structure. The hypograph set $\hyp_{\{0,1\}^n}(f)$ is a special case of the constraint set $\cS:=\{(x, t) \in \{0,1\}^n \times \bR: f_1(x) -f_2(x)  \ge \ell t\}$
with $ \ell \in \{0,1\}$. We also aim at extending our results to this more general set $\cS$. Finally, we propose an efficient algorithm to compute intersection cuts derived from $\cS$-free sets. To the best of our knowledge, intersection cuts have not  been applied directly to approximate problems with submodular and/or supermodular structures.

We implement intersection cuts within the \texttt{SCIP} solver \cite{bestuzheva2021scip} and test them on \maxcut,  \pbm, and  \bdopt problems. We show the strengths and weaknesses of intersection cuts under these different settings.

\subsection{Contributions}
Our primary contribution is the construction of hypograph-free sets. We show that a maximally hypograph-free set $\cC \times \bR$ can be lifted from  a maximal $\{0,1\}^n$-free set.  We also give an alternative construction of hypograph-free sets by exploiting  the submodularity. We relate the analytical properties of  $\sF$ to its combinatorial properties, which inherit those of the Lovász extension. We show that the epigraph $\epi(\sF)$ of $\sF$ is a hypograph-free set,  larger than the epigraph of the Lovász extension. However, unlike in the continuous setting, $\epi(\sF)$ is not maximally hypograph-free. We give necessary and sufficient conditions on maximal  hypograph-free sets that include $\epi(\sF)$.

The second contribution is  the computation of intersection cuts. We reduce the intersection cut separation problem to solving univariate nonlinear equations, which we achieve by a hybrid discrete Newton algorithm like \cite{Goemans}. We show that facets of $\epi(\sF)$ can be separated in strongly polynomial time. This implies that the (sub)-gradients required by the  Newton algorithm can be computed in a strongly polynomial time. The hybrid discrete Newton algorithm finds a zero point of a univariate nonlinear equation  in a finite number of steps. By contrast, the conventional bisection algorithm  only guarantees $\epsilon$-approximated solutions for $\epsilon > 0$.

Lastly, we  extend the previous findings to constraint sets involving an SS function. We show that any Boolean multilinear function is an SS function. This result yields intersection cuts for multilinear constraints in binary variables.

\subsection{Literature review}
  The work of Jack Edmonds \cite{edmonds2003submodular} plays a prominent role in the study of the combinatorial properties of submodular functions. We refer to \cite{schrijver2003combinatorial} for basic concepts and definitions. The convex envelope of a submodular function $f$ is its Lovász extension  \cite{Atamturk2021,lovasz1983submodular}. Submodular functions are a subclass of discrete convex functions, and we refer to \cite{murota1998discrete} for  more details about discrete convex analysis.

Valid inequalities for the hypographs of general submodular functions are called the \textit{base inequalities} \cite{nemhauser1978analysis}. For a class of special submodular functions, lifting procedures  \cite{ahmed2011maximizing,shi2022sequence}  can strengthen the base inequalities. The  base inequalities can be separated either using heuristics \cite{ahmed2011maximizing} or at integer points in a Benders-like framework \cite{coniglio2022submodular}. The method defined in \cite{Atamturk2021} combines valid inequalities for the submodular and supermodular components of an SS function. We refer to \cite{atamturk2020submodularity,atamturk2022supermodularity2,billionnet1985maximizing,bouhtou2010submodularity,han2022fractional,kilincc2021joint,rhys1970selection,shamaiah2010greedy,xusignomial,yu2023strong} for more details about the exploitation of  submodular/supermodular functions in mathematical programs. Supermodular polynomials in binary variables are defined and studied in \cite{billionnet1985maximizing,rhys1970selection}. The submodularity of the \dopt problem is exploited in \cite{SAGNOL2013258,shamaiah2010greedy}.

As mentioned above, intersection cuts generate valid inequalities for sets that are hard to optimize over. Gomory introduced the corner polyhedron \cite{gomory1969some}, and his celebrated mixed-integer cuts \cite{gomory1963algorithm} are special intersection cuts derived from split disjunctions \cite{Nemhauser1988}. The definition of intersection cuts for arbitrary set $\cS$ is due to \cite{dey2008,glover1973}. We refer to  \cite{andersen2010,andersen2007,basu2010,basu2019,conforti2015,conforti2011,cornuejols2015sufficiency,del2012relaxations,dey2008,richard2010group} for a more in-depth analysis.  The method given in \cite{towle2021intersection} can generate valid inequalities that cut off points outside $\cS$-free sets. We refer to \cite{andersen2007,belotti2015conic,klnc-karzan2015,kilinc-karzan2016,modaresi2015,modaresi2016} for relevant recent developments in mixed-integer conic programming.

For the cases where the nonconvexity of $\cS$ is not just due to integer variables, we refer to \cite{fischetti2018} for bilevel programs, \cite{bienstock2020outer} for outer-product sets, \cite{munoz2022maximal,munoz2022towards} for quadratic constraint sets, \cite{xusignomial} for signomial-term sets, and \cite{fischetti2020} for bilinear sets. The method given in \cite{serrano2019} constructs intersection cuts  for sets arising from factorable programs that contain DC functions \cite{khamisov1999optimization}.

Next, we discuss valid inequalities for polynomial programming, because we use polynomial programs in binary variables as a benchmark in our computational study. In \cite{bienstock2020outer}, intersection cuts approximate a nonconvex lifted set, namely the outer product set arising from  the extended formulation of a polynomial program. Lifted sets link decision variables to auxiliary variables representing (graphs of) monomials up to a given degree. We remark that in most combinatorial optimization problems,  decision variables are binaries. The polynomial program of interest is then a Boolean Multilinear Program (BMP).  The corresponding lifted set is the \textit{Boolean multilinear set} \cite{crama1993concave,fortet1960applications}, the convex hull of which is the so-called Boolean multilinear polytope. Valid inequalities for the Boolean multilinear polytope may be stronger than those for the  convex  hull of the outer product set. Various Gomory-Chvátal-based inequalities \cite{del2017polyhedral,del2018multilinear,del2020impact,del2022simple} are valid for the multilinear polytope. The separation and strength of these inequalities depend on  the hypergraph representing the underlying sparsity pattern of the multilinear set.

 We consider a constrained polynomial program, and assume that some  of its constraints are neither integrality constraints nor variable bound constraints. After lifting,  those constraints are linear constraints and  thus define a convex set $\cS_1$. The lifted set $\cS_2$ is nonconvex, and $\cS_1 \not\subseteq \cS_2$.  The polynomial program is then equivalent to linear optimization over $\conv(\cS_1 \cap \cS_2)$. However, in general, $\conv(\cS_1 \cap \cS_2) \ne \cS_1 \cap \conv(\cS_2)$, so the convexification of the lifted set may not yield an equivalent convex problem. To address this issue, one attempt is to directly consider $\conv(\cS_1 \cap \cS_2)$ and generate valid inequalities for it.  Some work in this sense exists for certain interesting special cases, e.g.~the intersection of multilinear sets with additional constraint sets such as cardinality constraints \cite{chen2023multilinear}. Another attempt is to consider  constraints in projected formulations, \eg in mixed-integer quadratically constrained quadratic programs \cite{saxena2011convex}. Since the  representation complexity  of the projected formulation  is smaller than that of the  extended formulation, this approach is also amenable to computation. In \cite{chmiela2022implementation,munoz2022maximal}, intersection cuts for the set defined by a quadratic constraint are derived. If additionally, the nonbasic variables of the LP relaxation are integer, the monoidal technique \cite{ChmielaMunozSerrano2023} can strengthen such intersection cuts.


However, generating valid inequalities for Boolean multilinear constraints, and, more generally, constructing $\cS$-free sets for nonlinear constraints on discrete variables, remain problems of considerable interest. In this paper, we look at these questions through a ``submodularity lens''.

\subsection{Notation}
 We let $ [n]:= \{1,\cdots,n\}$ for any positive integer $n$. We denote $\cB := \{0,1\}^n$, $\bar{\cB}: = [0,1]^n$. We assume that $[n]$ is equipped with the natural number order.
 $\mathbf{1}$ denotes the all-one vector, and $\mathbf{0}$ denotes the all-zero vector. For $S \subseteq [n]$, we denote by $\cha{S} \in \cB$ the characteristic vector of $S$. For vectors $a,b$, we let $(a,b)$ be their concatenation, and extend this notation naturally to the case where $b$ is a scalar. Given a set $\cD \subseteq \bR^n$ and a function $g: \cD \to \bR$, we adopt the usual notation $\epi_{\cD}(g), \gra_{\cD}(g), \hyp_{\cD}(g)$ to denote the epigraph, graph and hypograph of $g$ over $\cD$, respectively. For example, $\gra_{\cD}(g) := \{(x,t) \in\cD \times \bR: g(x) = t\}$. When  $\cD$ is omitted in the subscript, it is assumed to be $\bR^n$.
 For any set $\cS$, we denote by $\bd(\cS)$, $\ext(\cS)$, $\inter(\cS)$ its  boundary, extreme points, interior, respectively. When $\cS$ is not full-dimensional, $\relint(\cS), \relbd(\cS)$ denote its relative interior and relative boundary.

\subsection{Outline}
The rest of the paper is organized as follows. In \Cref{sec.prem}, we recall some preliminaries for intersection cuts. In \Cref{sec.extension}, we study extensions of submodular functions. In \Cref{sec.maxsub}, we study the hypograph-free sets for the submodular function. In \Cref{sec.ss}, we generalize the previous results for sets involving SS functions. In \Cref{sec.app}, we consider applications for intersection cuts to Boolean multilinear constraints and \bdopt. In \Cref{sec.sep}, we propose the hybrid discrete Newton algorithm for computing intersection cuts. In \Cref{sec.cresult}, we analyze the computational results.

\section{Intersection cut preliminaries}
\label{sec.prem}

In this section, we review the basic concept of intersection cuts. Assume that we are solving the optimization problem $\min_{z \in \cS}c z$. Given  a polyhedral outer approximation $\cP$  of a nonconvex set $\cS$, an LP relaxation is $\min_{z \in \cP}c z$. Then  an optimal \textit{relaxation point} $\relx{z}$ is  a vertex of $\cP$. An intersection cut is a particular cut that separates $\relx{z}$ from $\cS$.

\begin{definition}
Given $\cS \in \bR^p$, a closed set $\cC$ is called (convex) $\cS$-free, if  $\cC$ is convex and $\inter(\cC) \cap \cS = \varnothing$.
\end{definition}

The construction of intersection cuts \cite{conforti2014} needs two components: a corner polyhedron relaxation $\cR$ of $\cS$ with apex $\relx{z}$ ($\cR$ can be extracted from $\cP$), and an $\cS$-free set $\cC$ containing $\relx{z}$ in its interior. Then  an intersection cut  separates $\relx{z}$ from $\conv{(\cR \smallsetminus \inter(\cC))}$ (a set which, we note, contains $\cS$) as follows. The  half-space and ray representations of the corner polyhedron $\cR$ is as follows:
\begin{equation}
\label{eq.halfspace}
 \cR := \{z \in \bR^p: A(z - \relx{z}) \le 0\} =  \{z \in \bR^p:\exists \eta \in \bR^p_+ \textup{ }  z =\relx{z} +  \sum_{j =1}^{p} \eta_jr^j \},    
\end{equation}
 where $A$ is a $p$-by-$p$ invertible matrix, $r^j$ is the $j$-th column of $-A^{-1}$ and  an extreme ray of $\cR$.
 
 Define the \textit{step length}
 \begin{equation}
 \label{eq.steplen}
 \eta_j^\ast := \sup_{\eta_j \ge 0}\{\eta_j: \relx{z} + \eta_j r^j \in \cC\}.    
 \end{equation}
  The point $\relx{z}$ is separated by an intersection cut $$
	 \sum_{j =1}^{p} \frac{1}{\eta_j^\ast} A_j (z - \relx{z}) \le -1.$$
 Let $\cC, \cC^\ast$ be two $\cS$-free sets such that $\cC \subseteq\cC^\ast$. Then the intersection cut derived from $\cC^\ast$ dominates the intersection cut derived from $\cC$. This makes maximal $\cS$-free sets relevant in the study.

\section{Extensions of submodular functions}
\label{sec.extension}
 In this section, we study continuous extensions of submodular functions. W.l.o.g., we assume  in the sequel that for any submodular function $f$, $f(\mathbf{0}) = 0$ holds (by a translation of a constant). It is known that the Lovász extension \cite{lovasz1983submodular} extends $f$ from $\cB$ to $\bar{\cB}$. Based on this extension, we construct another  extension $\sF$ of $f$ defined over the entire space $\bR^n$, and study its analytical and combinatorial structures.

 We first look at some polyhedra associated with the submodular function $f$ \cite{Atamturk2021,schrijver2003combinatorial}.  Its \textit{extended polymatroid} is defined as $$\ep{f} :=\{s \in \bR^n:\forall x \in \cB,\, s  x \le f(x)\},$$ its convex hull of the epigraph $f$ over $\cB$ is defined as $$Q_f:= \conv(\epi_{\cB}(f)).$$ Recall that $\ext(\ep{f})$ are the vertices of $\ep{f}$, and we can further define a polyhedron
 \begin{equation}
	 \label{eq.ee}
	 \ee{f} := \{(x,t) \in \bR^{n+1}: \forall s \in \ext(\ep{f}),\, sx \le t\}.
 \end{equation}
In fact, $\ee{f}$ includes $Q_f$, because of the following lemma:
 \begin{lemma}[\cite{Atamturk2021}]
\label{lem.at}
$ Q_f =  \ee{f} \cap (\bar{\cB} \times \bR).$
\end{lemma}

Therefore, $x \in \bar{\cB}$ defines trivial facets of $Q_f$, and non-trivial facets of $Q_f$ are  $sx \le t$, where $s$ is a vertex of $\ep{f}$.

These polyhedra in turn give rise to some functions associated with $f$. A convex function $g $ is a \textit{convex underestimating function} of $f$ over $\cB$, if for all $x \in \cB$, $g(x) \le f(x)$.
The \textit{convex envelope} $\conve_{\cB}(f)$ is thus the maximal convex underestimating function of $f$ over $\cB$.  Since $Q_f$ is the epigraph of $\conve_{\cB}(f)$, by \Cref{lem.at}, $$\conve_{\cB}(f): \bar{\cB} \to \bR: x \to \max_{s \in \ext(\ep{f})}sx,$$ where its domain is $\bar{\cB}$.  We remark that  $\conve_{\cB}(f)$ is equivalent to the Lovász extension of $f$ \cite{Atamturk2021}. We will show that the cardinality $|\ext(\ep{f})|$ is not polynomial to $n$. Thereby, when computing $\conve_{\cB}(f)$, it is inefficient to evaluate all $sx$ for $s \in \ext(\ep{f})$. In fact, the value and the (sub)-gradients of  $\conve_{\cB}(f)$ on  $\bar{\cB}$ can be computed in a strongly polynomial time \cite{Atamturk2021}.

We define the envelope of $f$ extended to $\bR^n$ as  $$\sF: \bR^n \to \bR: x \to \max_{s \in \ext(\ep{f})}sx.$$  We note that $\sF$ simply enlarges the domain of $\conve_{\cB}(f)$ from $\bar{\cB}$ to $\bR^n$. This extension is algebraically simple, but analytical properties of $\sF(x)$ outside $\bar{\cB}$ will be studied in further detail. We find that
$\ee{f}$ is the epigraph of $\sF$,  \ie $\ee{f} = \epi(\sF)$, so $\sF$ is a convex function. Since every facet $sx \le t$ of $\ee{f}$  is in one-to-one correspondence to a linear underestimator function $sx$ of $\sF$, we call $\ee{f}$ the \textit{extended envelope epigraph}.

A fundamental question on $\sF$ is how to compute its value and (sub)-gradients efficiently, because  this is crucial in constructing intersection cuts. Since  the Lovász extension $\conve_{\cB}(f)$  is a restriction of $\sF$ on the hypercube $\bar{\cB}$,  we can compute $\sF$ efficiently (in a strongly polynomial time) on $\bar{\cB}$. In the following, we will show how to extend this method to compute $\sF$ over the entire space $\bR^n$. This  extension requires us to study the properties of  $\sF$ and $\ee{f}$.

We first look at combinatorial structures associated with the facets of $\ee{f}$. Recall that a permutation $\pi$ on $[n]$ is a bijective map from $[n]$ to itself. The map $\pi(i) \in [n]$ is the image of   an element $i \in [n]$ under this permutation. We denote by $\Pi$ the set of permutations on  $[n]$. We define the following sets and vectors related to permutations.

\begin{definition}
\label{def.perm}
 Given a permutation $\pi \in \Pi$ and an integer $i \in \{0, \cdots, n\}$, define $\pi([i]) := \{\pi(1),\cdots,\pi(i)\}$ ($\pi([0]) := \varnothing$), and  define $v^i(\pi) := \cha{\pi([i])}$. Define the map $\sigma: \Pi \to \bR^n$ such that it satisfies $\sigma(\pi)_{\pi(i)}= f(v^i(\pi)) -f(v^{i-1}(\pi))$  for all $\pi \in \Pi$ and for all $i \in [n]$.
\end{definition}

  The set of vertices  $\ext(\ep{f})$ is  the image of  $\Pi$ under the map $\sigma$.

\begin{lemma}[\cite{edmonds2003submodular}]
\label{lem.subvert}
$\sigma(\Pi) = \ext(\ep{f})$.
\end{lemma}

Every  permutation $\pi \in \Pi$ induces a vertex $\sigma(\pi)$ of $\ext(\ep{f})$ through the map $\sigma$, so the cardinality of $\ext(\ep{f})$ is $n!$ (not polynomial to $n$). The above lemma  shows that every facet of $\ee{f}$ (a non-trivial facet of $Q_f$) is given as $\sigma(\pi)x \le t$, and every linear underestimator of $\sF$ is given as $\sigma(\pi)x$.

 \begin{proposition}
 \label{prop.supportpoint}
Given a permutation $\pi \in \Pi$, for all $i \in [n] \cup \{0\}$, the facet-defining inequality $\sigma(\pi) x \le t$ is supported by $ \bigl(v^i(\pi), f(v^i(\pi))\bigl)$, \ie $\sigma(\pi) v^i(\pi) = f(v^i(\pi))$.
\end{proposition}
\begin{proof}
\begin{equation*}
    \sigma(\pi)v^i(\pi) = \sum_{ j \in [i] } \sigma(\pi)_{\pi(j)}
    = \sum_{ j \in [i] } \Bigl( f(v^j(\pi)) -f(v^j(\pi)) \Bigl)
	 = f^i(v^i(\pi))  - f(0)
    = f(v^i(\pi))
,\end{equation*}
 where the first equation follows from \Cref{def.perm}, the second equation follows from \Cref{lem.subvert}, and the last two equations follow from the expansion of the sum.
\end{proof}
 
  Conversely to \Cref{prop.supportpoint}, given a point in the graph of $f$, we can construct all the facets supported by it.
 
  \begin{corollary}
  \label{cor.point}
   For a point $v \in \cB$, let $\iota$ be the number of ones in $v$. If  a permutation $\pi \in \Pi$ satisfies that $v = v^\iota(\pi)$, then $ \left(v, f(v)\right)$ supports the facet-defining inequality $\sigma(\pi) x \le t$ of $\ee{f}$.
  \end{corollary}

At the moment, we find that  one can easily obtain the facial structure of $\ee{f}$ from that of $Q_f$. We ask how to separate facets of $\ee{f}$.
Since $\ee{f}$ is the epigraph of  $\sF$, the shape of $\ee{f}$ is determined by $\sF$, so it suffices to look at $\sF$.

From a convex analysis perspective, the  nonsmooth polyhedral function $\sF$ is the maximum of a set of linear functions, so it is  convex and positive homogeneous of degree 1. This means that $\sF$ is subdifferentiable \cite{hiriart2004fundamentals}. Moreover, $\sF$ has the following analytical properties.
 \begin{proposition}
 \label{prop.sub}
For all $x',x \in \bR^n$ and all $s \in \partial \sF(x')$, $\sF(x')=sx'$ and $\sF(x)\ge  sx$.  Moreover, $ \partial \sF(x') =\conv(\argmax_{s \in \ext(\ep{f})} s x')$.
\end{proposition}
\begin{proof}
As $\sF(x)= \max_{s \in \ext(\ep{f})} s x$, $\sF$ is the maximum of a set of linear functions. This implies that it is positive-homogeneous of degree-1 and convex, and it is easy to show the other results.
\end{proof}

Given $\relx{x} \in \bR^n$, the evaluation of $\sF(\relx{x})$ is called the \textit{extended polymatroid vertex maximization problem}, as by definition $\sF(\relx{x}) $ equals
\begin{equation}
\label{eq.sep}
    \max_{s \in \ext(\ep{f})} s \relx{x}.
\end{equation}

  By \Cref{prop.sub}, an optimal solution $s^\ast$ is a subgradient of $\sF$ at $\relx{x}$ (\ie $s^\ast \in \partial \sF(\relx{x}))$.  By \Cref{lem.subvert},  $ \max_{s \in \ext(\ep{f})} s \relx{x} = \max_{\pi \in \Pi} \sigma(\pi)\relx{x}$, so \eqref{eq.sep} asks for a permutation $\pi^\ast$ that maximizes $\sigma(\pi^\ast)\tilde{x}$.  One of the main findings in this section is an algorithm to solve \eqref{eq.sep}.

To tackle \eqref{eq.sep}, we  look at a related relaxed problem, namely the \textit{extended polymatroid maximization problem}, that is well studied:
\begin{equation}
\label{eq.sep2}
    \max_{s \in \ep{f}} s \relx{x}.
\end{equation}
If $\relx{x} \ge 0$,   a strongly polynomial time   \textit{sorting algorithm}  can  solve the extended polymatroid maximization  \cite{edmonds2003submodular}: Let $\pi^\ast \in \Pi$ be a permutation  such that $\relx{x}_{\pi^\ast(1)}\ge\cdots  \ge \relx{x}_{\pi^\ast(n)} $,  then an optimal solution to \eqref{eq.sep2} is  $\sigma(\pi^\ast)$.

We note that the vertices $ \ext(\ep{f})$ are a finite set, so \eqref{eq.sep} is always bounded; $\ep{f}$ is the Minkowski sum of $\ext(\ep{f})$ and a set of recession rays, so $\ep{f}$ is unbounded. This means that \eqref{eq.sep2} can be unbounded.
\begin{lemma}
\label{lem.notequi}
[\cite{Atamturk2021,edmonds2003submodular}]
\label{lem.sort}
When $\tilde{x} \ge  0$,  the optimum of  \eqref{eq.sep2} must be a vertex, and \eqref{eq.sep} is equivalent to \eqref{eq.sep2}; when $\tilde{x}$ has some negative entries,  \eqref{eq.sep} is unbounded and is not equivalent to  \eqref{eq.sep2}.
\end{lemma}
 Therefore, \eqref{eq.sep} is not equivalent to \eqref{eq.sep2} for $\title{x} \in \bR^n \smallsetminus \bR^n_{+}$. However, we can show that the sorting algorithm also solves the problem \eqref{eq.sep} for any case.

\begin{proposition}
\label{prop.out}
The output of the sorting algorithm is optimal to the extended polymatroid vertex
maximization problem \eqref{eq.sep}.
\end{proposition}
\begin{proof}

Let $\pi^\ast$ be the permutation found by the sorting algorithm. By \Cref{lem.subvert}, $\sigma(\pi^\ast)$ is in $\ext(\ep{f})$ and hence  a feasible solution to \eqref{eq.sep}.  Next, we prove the optimality of $\sigma(\pi^\ast)$. Let $d := \min_{i \in [n]} \relx{x}_i$, then the translated vector $\relx{x} - d\mathbf{1}  = (\relx{x}_i-d)_{i \in [n]}  \ge \mathbf{0}$. The following inequalities hold:
\begin{equation}
\label{eq.many}
\begin{split}
   \sigma(\pi^\ast) \relx{x} & \le \max_{s \in \ext(\ep{f})} s \relx{x} = \max_{s \in \ext(\ep{f})} s (\relx{x}-d \mathbf{1} +d \mathbf{1}) 
   \\ &\le  \max_{s \in \ext(\ep{f})} s  (\relx{x} - d \mathbf{1}) +  \max_{s \in \ext(\ep{f})} s(d \mathbf{1}).
\end{split}
\end{equation}
It is easy to show that $(\relx{x} - d \mathbf{1} )_{\pi^\ast(1)}\ge\cdots  \ge (\relx{x} - d \mathbf{1})_{\pi^\ast(n)} $. As $\relx{x} - d \mathbf{1} \ge \mathbf{0}$,  by \Cref{lem.notequi} and the sorting algorithm, $ \sigma(\pi^\ast) (\relx{x} - d \mathbf{1}) = \max_{s \in \ep{f}} s  (\relx{x} - d \mathbf{1}) = \max_{s \in \ext(\ep{f})} s  (\relx{x} - d \mathbf{1}) $. It follows from \Cref{prop.supportpoint} that $\sigma(\pi)v^n(\pi) = \sigma(\pi)\mathbf{1} = f(\mathbf{1})$. As the entries of $d \mathbf{1}$ are identical,   for any $\pi \in \Pi$,  $\sigma(\pi) (d \mathbf{1}) = d f(\mathbf{1})$. Therefore, for any $s \in  \ext(\ep{f})$,  $s (d \mathbf{1}) = d f(\mathbf{1})$, so $\sigma(\pi^\ast) (d \mathbf{1}) = \max_{s \in \ext(\ep{f})} s(d \mathbf{1})$. Looking at the inequalities \eqref{eq.many}, they become equations, because
\begin{equation*}
   \sigma(\pi^\ast) \relx{x} \le \max_{s \in \ext(\ep{f})} s \relx{x} \le\sigma(\pi^\ast)   (\relx{x} - d \mathbf{1}) +  \sigma(\pi^\ast)(d \mathbf{1})=  \sigma(\pi^\ast) \relx{x}.
\end{equation*}
Therefore, $\sigma(\pi^\ast)$ is an optimal solution to \eqref{eq.sep}.
\end{proof}

Given $\tilde{x} \in \bR^n$, the sorting algorithm outputs a permutation on it. The sorting algorithm  is translation-invariant, \ie translating each entry of $\tilde{x}$ by the same value does not change the output permutation. A byproduct of \Cref{prop.out} is that the translation invariance implies the ray-linearity of $\sF$.

\begin{corollary}
    Let $\tilde{x} \in \bR^n$, then $\sF$ is linear on $\tilde{x} + \lambda \mathbf{1}$ w.r.t. $\lambda \in \bR$.
\end{corollary}

 We look at the boundary of $\ee{f}$. By \Cref{prop.supportpoint} and \Cref{cor.point}, for all $x \in \cB$, the point $(x,f(x))$  supports  some facets of $\ee{f}$.

\begin{theorem}
\label{thm.bdeef}
$\ee{f} \cap \hyp_{\cB}(f) = \gra_{\cB}(f) \subseteq \bd(\ee{f})$.
\end{theorem}
\begin{proof}
We consider a point $v \in \cB$ and look at the line $ \ell = \{(v,t): t \in \bR\}$. It can be separated into the restricted epigraph $\ell_+:=\{(v,t):f(v) \le t\}$ and  the restricted hypograph $\ell_-:=\{(v,t):f(v) \ge t\}$, as $\ell_+ \cap \ell_- = (v, f(v))$ and $ \ell = \ell_+ \cup \ell_-$.
First,  we know that, by  definition of $Q_f$ and \Cref{lem.at},   $\ell_+ \subseteq Q_f \subseteq \ee{f}$. Second, by \Cref{prop.supportpoint},  the point $(v, f(v))$ supports some facets of $\ee{f}$, so the point $(v,t)$ with $t < f(v)$ is separated by these facets from  $\ee{f}$. Thereby, we know that $\ell_- \cap \ee{f} = \{(v, f(v))\}$. To summarize, we know that $\ee{f} \cap \ell = \ell_+$ and $(v, f(v)) \in \bd(\ee{f})$. As $\gra_{\cB}(f) = \cup_{v \in B} \{(v,f(v))\}$, we have that  $ \gra_{\cB}(f)  \subseteq \bd(\ee{f})$. As the hypograph $\hyp_{\cB}(f) = \cup_{v \in \cB} \{(v,t):f(v) \ge t\} $ (union of restricted hypographs), we have that $\ee{f} \cap \hyp_{\cB}(f) = \gra_{\cB}(f)$.
\end{proof}

 As already mentioned, $\sF$ is convex and $\ee{f} =\epi(\sF) $, so $\sF$ is also a continuous extension of $f$. As $\ee{f}$ includes $Q_f$, $\sF$ further extends $\conve_{\cB}(f)$ (the Lovász extension).
 
We now understand the facial structure of $\ee{f}$, which will help us construct hypograph-free sets.
We also know how to compute the value and subgradients of $\sF$ at any point in $\bR^n$, which is important for constructing intersection cuts.

\section{Hypograph-free sets for submodular functions}
\label{sec.maxsub}
 In this section, we construct two types of hypograph-free sets for the  submodular function $f$.

 First, we show that one can lift a maximal $\cB$-free set into a  maximal $\hyp_{\cB}(f)$-free  set.
 \begin{theorem}
\label{thm.bin}
Let $f:\cB \to \bR$ be an arbitrary function, and let $\cK$ be a maximal $\cB$-free set in $\bR^n$. Then  $\cC:=\cK \times \bR$ is a maximal $\hyp_{\cB}(f)$-free set.
\end{theorem}
\begin{proof}
We note that $\inter(\cC)= \inter(\cK) \times \bR$.
It is easy to show that $\cC$ is $\hyp_{\cB}(f)$-free, since $\inter(\cC) \cap \hyp_{\cB}(f) = \varnothing$.
Assume that there exists a $\hyp_{\cB}(f)$-free set $\cC'$ including $\cC$. Then  the recession cone of $\cC'$ must include  that of $\cC$, so $\cC' = \cK' \times \bR$ for some closed convex set $\cK'$ including $\cK$. Moreover, $\cK'$ must be a  $\cB$-free set, otherwise, there exists a point $x \in \cB \cap \inter(\cK')$ such that $(x,f(x)) \in \inter(\cK') \times \bR = \inter(\cC')$. However, since $\cK$ is maximally $\cB$-free, this implies that $\cK = \cK'$. As a result, $\cC = \cC'$, so $\cC$ is maximal.
\end{proof}

 This construction does not rely on any structure of $f$, as it just lifts a $\cB$-free set. For any $j \in [n]$, the simple lifted split $ \{x \in \bR^n: 0 \le x_j \le 1\} \times \bR$ is a maximal $\hyp_{\cB}(f)$-free set. We next construct  $\hyp_{\cB}(f)$-free sets using the submodularity, for both theoretical and computational interests.

We show that the extended envelope epigraph is a hypograph-free set.
\begin{proposition}
\label{cor.free}
$\ee{f}, Q_f$ are $\hyp_{\cB}(f)$-free sets.
\end{proposition}
\begin{proof}
Since $\gra_{\cB}(f) \subseteq \bd(\ee{f})$, we conclude that  $\ee{f} \cap \hyp_{\cB}(f) \subseteq  \bd(\ee{f})$ and hence $\inter(\ee{f}) \cap \hyp_{\cB}(f) = \varnothing$. Additionally, $\ee{f}$ is convex and hence $\hyp_{\cB}(f)$-free. As $Q_f \subseteq \ee{f}$, $Q_f$ is $\hyp_{\cB}(f)$-free set.
\end{proof}

It is known that for a convex function, its maximal hypograph-free set is its epigraph. However, for the submodular function $f$,  we will show that its extended epigraph $\ee{f}$ is not a maximal hypograph-free set. A high-level way to test the maximality of $\ee{f}$ is as follows. The set $Q_f$ is  the convex hull of $\epi_{\cB}(f)$. Geometrically,   $Q_f$ is the ``minimal'' convex set including $\epi_{\cB}(f)$. This is a conflict as we aim to obtain an inclusion-wise ``maximal'' $\hyp_{\cB}(f)$-free set. We can remove some facets from $Q_f$ and thus enlarge this polyhedron. After removing trivial facets of  $Q_f$,  the enlarged polyhedron is the  extended envelope epigraph $\ee{f}$. However, this enlargement is still not enough.

 We look at a concrete characterization of the ``correct'' enlarging of $\ee{f}$.
The following fundamental theorem gives a sufficient and necessary condition on (maximal) hypograph-free sets including $\ee{f}$.

\begin{theorem}
\label{thm.free}
Let $\cC$ be a full-dimensional  closed convex set in $\bR^{n+1}$ including  $\ee{f}$. Then  $\cC$ is a  $\hyp_{\cB}(f)$-free set if and only if $\cC$ is $\gra_{\cB}(f)$-free. Moreover, $\cC$ is a maximal $\hyp_{\cB}(f)$-free set if and only if $\cC$ is a polyhedron and there is at least one point of  $\gra_{\cB}(f)$ in the relative interior of each facet of $\cC$.
\end{theorem}
\begin{proof}
We note that by \Cref{thm.bdeef}, $\gra_{\cB}(f) \subseteq \bd(\ee{f}) \subseteq  \ee{f} \subseteq \cC$. Thereby,  $\gra_{\cB}(f) \cap \inter(\cC) = \varnothing$ (\ie $\cC$ is $\gra_{\cB}(f)$-free) if and only if $\gra_{\cB}(f) \subseteq \bd(\cC)$.

We consider the  $\cS$-freeness first. We prove the forward direction. Assume that $\cC$ is a  $\hyp_{\cB}(f)$-free set. Suppose, to aim at a contradiction, that there exists a point $(v,f(v)) \in \inter(\cC) \cap \gra_{\cB}(f)$. Then  there exists a sufficiently small $\epsilon > 0$ such that $(v, f(v) -\epsilon) \in \inter(\cC)$, but $(v, f(v) -\epsilon)  \in  \hyp_{\cB}(f)$, which leads to a contradiction. We prove the reverse direction. Assume that $\cC$ is $\gra_{\cB}(f)$-free.  Suppose, to aim at a contradiction, that there exists a point $(v, f(v) - \delta) \in \inter(\cC)$ with $v \in \cB$ and $\delta> 0$. As, for some $\epsilon > 0$, $(v, f(v) + \epsilon) \subseteq \inter(\ee{f}) \subseteq \inter(\cC)$, by convexity of $\cC$, $(v,f(v)) \in \inter(\cC)$, which leads to a contradiction. This implies that $\cC$ is  $\hyp_{\cB}(f)$-free if and only if $\gra_{\cB}(f)$-free (or $\gra_{\cB}(f) \subseteq \bd(\cC)$).

  We consider the maximality next. Due to \cite{basu2010maximal}, a full-dimensional lattice-free set is maximal if it is a polyhedron and there is at least one lattice point in the relative interior of each facet. As $\gra_{\cB}(f)$ is a finite set, the proof strategy is similar although it is not a subset of any lattice. Then  the result follows.
\end{proof}

The above theorem is purely geometrical. Since submodular functions are combinatorial objects, we translate this theorem into a combinatorial language. We first define a combinatorial object in the Boolean hypercube $\cB$.

\begin{definition}
Let $x^0, x^1,  \cdots, x^n$ be $n+1$ distinct points of $\cB$.  They are called  \textit{monotone}, if $\mathbf{0} = x^0 < x^1 < \cdots < x^n = \mathbf{1} $.  We call the corresponding ordered set $(x^0, \cdots, x^n) \subseteq \cB$  a \textit{monotone chain} in $\cB$.
\end{definition}

 Therefore,  we use a monotone chain to represent a set of monotone points. Then  we have the following observation.

\begin{proposition}
\label{prop.one2one}
The set of monotone chains is in one-to-one correspondence to the set $\Pi$ of permutations  via the map $V$ defined as follows: for all $\pi \in \Pi$, $V(\pi):= (v^i(\pi) \;|\; i\in\mathcal{N}\cup\{0\})$.
\end{proposition}
\begin{proof}
By \Cref{prop.supportpoint}, since $\varnothing = \pi([0]) \subsetneq \cdots \subsetneq  \pi([n]) = [n]$, by \Cref{def.perm}, $\mathsf{0} = v^0(\pi) < \cdots < v^n(\pi) = \mathsf{1}$, so $V(\pi)$ is a monotone chain. Conversely, given a monotone chain $(x^0, \cdots, x^n)$, its inverse map $\pi$ exists and satisfies that $\pi(0) = 0$; and for all $i \in [n]$, $\pi(i)$ is the index  of the unique non-zero entry of $x^i - x^{i-1}$.
\end{proof}

We find that  permutations and monotone chains are indeed equivalent. We note that any $n+1$ distinct points from $\gra_{\cB}(f)$ are affinely independent in $\bR^{n+1}$ and hence support a hyperplane in $\bR^{n+1}$. Thereby, we can infer from \Cref{prop.supportpoint} and \Cref{prop.one2one} that \begin{corollary}
 If $(x^0, \cdots, x^n)$ is a monotone chain in $\cB$, then distinct points $ (x^0,f(x^0)), \cdots,  (x^n,f(x^n))$ of $\gra_{\cB}(f)$ define (or support) a facet of the extended envelope epigraph $\ee{f}$.
\end{corollary}

We say that this monotone chain \textit{induces the facet}. In fact, we find that facets of $\ee{f}$, permutations on $[n]$, and monotone chains in $\cB$ are in one-to-one correspondence. Therefore, we can view them as the same objects. Especially,  \Cref{prop.one2one} relates permutations and monotone chains. We give the following characterization of permutations on $[n]$.

\begin{definition}
A subset $\Pi'$ of permutations of $\Pi$   is called  a cover, if $\bigcup_{\pi \in \Pi'} V(\pi) = \cB$;  moreover, $\Pi'$ is  called a minimal cover, if additionally, for all $\pi \in \Pi'$, $V(\pi) \smallsetminus \bigcup_{\pi' \in \Pi': \pi' \ne \pi} V(\pi')$ is not empty.    
\end{definition}

We want to enlarge $\ee{f}$ by removing its facets, this is equivalent to removing permutations  from $\Pi$.  Let  $\Pi'$ be a subset  of permutations of $\Pi$, and $\cC(\Pi') := \{(x,t): \forall \pi \in \Pi', \, \sigma(\pi) x \le t\}$ denotes the relaxation of the extended envelope epigraph induced by $\Pi'$. It is obvious that $\ee{f} = \cC(\Pi) \subseteq \cC(\Pi')$ for any $\Pi' \subseteq \Pi$. The following corollary translates \Cref{thm.free} in a combinatorial language.
\begin{corollary}
\label{cor.permax}
Let $\Pi'$ be a  subset  of permutations of $\Pi$. $\cC(\Pi')$ is $\hyp_{\cB}(f)$-free if and only if $\Pi'$ is a cover. $\cC(\Pi')$ is maximally $\hyp_{\cB}(f)$-free  if and only if $\Pi'$ is a minimal cover.
\end{corollary}
\begin{proof}
First, we note that $C(\Pi')$, as a relaxation of $\ee{f}$ includes $\gra_{\cB}(f)$. Next, we assume  that $\Pi'$ is a cover. Then  points of $\gra_{\cB}(f)$ support facets of $C(\Pi')$. By \Cref{thm.free}, $C(\Pi')$ is $\hyp_{\cB}(f)$-free if and only if it is a cover. Finally,  $\Pi'$ is a minimal cover, if and only if then each facet of $C(\Pi')$ has a point of $\gra_{\cB}(f)$ in its interior. By \Cref{thm.free}, the later is equivalent to that $C(\Pi')$ is  maximally $\hyp_{\cB}(f)$-free.
\end{proof}

We now can disprove the maximality $\ee{f}$ by a counter-example. Thanks to the \Cref{cor.permax}, we can use a counting argument to show that we can remove facets from $\ee{f}$. This results in a new enlarged $\hyp_{\cB}(f)$-free  polyhedron.
 
 \begin{proposition}
  $\ee{f}$ is  not maximally hypograph-free.
 \end{proposition}
 \begin{proof}
 It suffices to find a counter-example.
  Consider $n=3$, $\cB = \{0,1\}^3$, there are 6 permutations, and 6 monotone chains (see \Cref{fig.ham}). We assume that, in a non-degenerate case, the associated extended envelope epigraph $\ee{f}$ has 6 facets induced by 6 chains respectively. The vertices $(0,0,0)$ and $(1,1,1)$ are visited by all the chains, while the other vertices are visited twice each. Therefore, a chain cannot  ``exclusively'' visit a vertex, so the corresponding facet  cannot contain one point of  $\gra_{\cB}(f)$ in its relative interior. In fact, we can remove some facets from  the extended envelope epigraph. We keep three chains:
\begin{eqnarray*}
 \left(\,(0,0,0) , (0,0,1) , (0,1,1) , (1,1,1)\,\right), \\
 \left(\,(0,0,0) , (0,1,0) , (1,1,0) , (1,1,1)\,\right), \\
 \left(\,(0,0,0) , (1,0,0) , (1,0,1) , (1,1,1)\,\right).
\end{eqnarray*}
 These chains induce 3 facets such that at least one point of  $\gra_{\cB}(f)$ is in the relative interior of each facet and each point of $\cB$ is in these 3 facets, so the polyhedron defined by these 3 facets is a $\hyp_{\cB}(f)$-free set larger than $\ee{f}$.
 
 \end{proof}

\begin{figure}[t]
\centering
\captionsetup{justification=centering}
\includegraphics[width = 0.3\columnwidth]{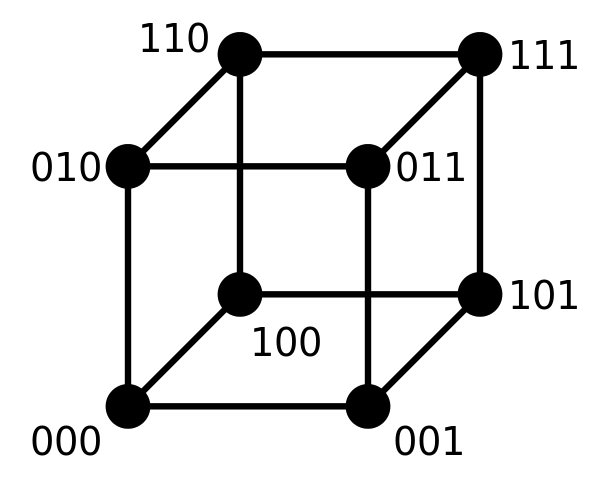}
\caption{$\cB = \{0,1\}^3$}
\label{fig.ham}
\end{figure}

We explain the hardness to enlarge $\ee{f}$. We build a bipartite graph $G:=(\cB \cup \Pi, E)$. An edge $e$ of $E$ connects a vertex $v \in \cB$ to a permutation $\pi \in \Pi$ if  $v \in V(\pi)$. Then, a minimal cover is a subset $\Pi'$  of  $\Pi$ such that i) each vertex of $\cB$ is incident to at least one permutation in $\Pi'$; ii) each permutation in $\Pi'$ is incident to a vertex of $\cB$ that no other permutation in  $\Pi'$ is incident to. As $|\cB| = 2^n$ and $|\Pi|=n!$, the size of such a graph is not polynomial to $n$. Therefore, one may need additional structural information  to enlarge $\ee{f}$ efficiently.

We relax the submodular maximization problem \eqref{eq.milp} via a polyhedral outer approximation $\cP$ of $\hyp_{\cB}(f)$. Let $X$ be the orthogonal projection of $\cP$ on $x$-space. We remark that, within a branch-and-cut algorithm, $X$ might be within a low-dimensional face of $\bar{\cB}$. Let $\relx{z}:=(\relx{x}, \relx{t})$ be a solution to the LP relaxation $\max_{(x,t) \in \cP} t$. We assume that $\relx{x} \notin \cB$, otherwise, $\relx{x}$ is already an optimal solution to \eqref{eq.milp}. The polyhedral outer approximation $\cP$ gives rise to a piece-wise linear concave overestimating function of $f$ over $X$: $\bar{f}(x) := \max_{(x,t) \in \cP} t$,
such that $\max_{(x,t) \in \cP} t = \max_{x \in X}\bar{f}(x)$. We then have the following observation.

\begin{proposition}
\label{prop.in}
Assume that $f$ is not affine over $X$, and let $x^\ast \in \relint(X)$. Then $\bar{f}(x^\ast) > \sF(x^\ast)$, \ie $(x^\ast, \bar{f}(x^\ast)) \in \inter(\ee{f})$.
\end{proposition}
\begin{proof}
As $\bar{f}$ is concave overestimator of $f$ over $X$ and $\sF$ is convex  underestimator of $f$ over $X$, $\bar{f} \ge \sF$ over $X$.  Suppose, to aim at a contradiction, that  $\bar{f}(x^\ast) = \sF(x^\ast)$. Define a concave function $g:= \bar{f} - \sF$, then for all $x \in X$, $g(x) \ge 0$, and $g(x^\ast) = 0$. By its concavity, there exists  an affine overestimating function $a$ of  $g$, such that $g(x^\ast) = a(x^\ast) = 0$, and for all $x \in X$, $0 \le  g(x) \le a(x)$. As $x^\ast \in \relint(X)$, the affinity of $a$  implies that $a = g = 0$ over $X$, \ie $\bar{f}= \sF$ over $X$. So $f$ is concave and convex over $X$ and thus affine over $X$, which is a contradiction.
\end{proof}

  The measure of the relative boundary $ \relbd(X)$ is zero, so we can assume that a mild \textit{relative interior condition} that $\relx{x} \in \relint(X)$ holds with probability one. Then  the relaxation point  $\relx{z}$ is  in the relative interior of the extended envelope epigraph with probability one.  
 
 \section{Hypograph-free and superlevel-free sets for SS functions}
\label{sec.ss}

This section considers hypograph and  superlevel sets for an SS function $f:=f_1 - f_2$, where $f_1$ and $f_2$ are two submodular functions. This  generalizes our previous results for the hypograph set of the submodular function, and thus one can generate intersection cuts for a larger family of discrete nonconvex sets.

More specifically, we consider the following nonconvex set
\begin{equation}
    \label{eq.cs}
	 \cS := \{(x, t) \in \cB \times \bR: f(x)  \ge \ell t\},
\end{equation}
with $\ell \in  \{ 0, 1\}$. Given a relaxation point $(\relx{x}, \relx{t}) \notin \cS$,
we want to find cutting planes separating this point from $\cS$.

Let $\mathsf{F_1} := \max_{s \in \ep{f_1}} sx$ and $\mathsf{F_2} := \max_{s \in \ep{f_2}} sx$ be extended envelopes of $f_1, f_2$, respectively. As $\mathsf{F_1}$ (resp. $\mathsf{F_2}$) is a convex extension of $f_1$ (resp. $f_2$),  we have that $\cS = \{(x, t) \in \cB \times \bR: \mathsf{F_1}(x) - \mathsf{F_2}(x)  \ge \ell t\}.$ By relaxing $\cB$ to $\bR^n$,  a (nonconvex) continuous outer approximation of $\cS$ is
\begin{equation}
    \label{eq.csb}
	 \bar{\cS} := \{(x, t) \in \bR^n \times \bR: \mathsf{F_1}(x) - \mathsf{F_2}(x)  \ge \ell t\}.
\end{equation}
Moreover, for all $x \in \cB$, $(x,t) \in  \bar{\cS}$ if and only if $(x,t) \in  \cS$.

\textbf{Special cases. } When $ \ell = 1$, $\cS$ is the hypograph of the SS function $f$; when $ \ell = 0$,  $\cS$ is the 0-superlevel set  of the SS function $f$. Setting $f_2 = 0$ and $ \ell = 1 $, the set $\cS$  becomes the hypograph
$
    \{(x, t) \in \cB \times \bR: f_1(x) \ge t \},
$
which is studied in the previous section. Setting $f_1 = 0$, the relaxed set $\bar{\cS}$ becomes
$
    \{(x, t) \in \cB \times \bR: \mathsf{F_2}(x) \le -\ell t\}.
$
If $(\relx{x}, \relx{t}) \notin \bar{\cS}$, since $\mathsf{F_2}(x) \ge \gamma^\ast x$ and $\mathsf{F_2}(\relx{x}) = \gamma^\ast \relx{x}$  for any $\gamma^\ast \in \partial{\mathsf{F_2}}(\relx{x})$, then the simple outer approximation cut $ \gamma^\ast x \le -\ell t$ is a valid inequality for $\bar{\cS}$ (hence for $\cS$).

In general, we should separate intersection cuts specifically for SS functions. Let $\gamma^\ast \in \partial{\mathsf{F_2}}(\relx{x})$ be a solution to \eqref{eq.sep} associated with $f_2$, and we define the  set
\begin{equation}
\label{eq.c}
	 \cC_{\relx{x}} := \{(x, t) \in \bR^n \times \bR:  \mathsf{F_1}(x) - \gamma^\ast x   \le\ell t\}.
\end{equation}
 The following proposition gives $\cS$-free sets.

\begin{proposition}
\label{prop.dsfree}
The set $\cC_{\relx{x}}$ is an $\cS$-free set. Moreover, if $(\relx{x}, \relx{t}) \notin \bar{\cS}$, then $\cC_{\relx{x}}$ does not contain $\relx{x}$ in its interior.
\end{proposition}
\begin{proof}
We first prove that $\cC_{\relx{x}}$ is $\bar{\cS}$-free. By definition, $\gamma^\ast x \le \mathsf{F_2}(x)$, which implies that $ \mathsf{F_1}(x) - \gamma^\ast x  \ge \mathsf{F_1}(x) - \mathsf{F_2}(x)  $. Therefore, for $(x,t) \in \inter(\cC_{\relx{x}})$, we have that $\ell t >   \mathsf{F_1}(x) - \gamma^\ast x  \ge \mathsf{F_1}(x) - \mathsf{F_2}(x) $, which implies that $(x,t) \notin \bar{\cS}$. Hence, $\inter(\cC_{\relx{x}}) \cap \bar{\cS} = \varnothing$. Additionally, $\cC_{\relx{x}}$ is convex. These two facts imply that $\cC_{\relx{x}}$ is $\bar{\cS}$-free. Since $\cS \subseteq \bar{\cS}$, $\cC_{\relx{x}}$ is also an $\cS$-free set. Next, assume that $(\relx{x}, \relx{t}) \notin \bar{\cS}$, then $ \ell \relx{t} >  \mathsf{F_1}(\relx{x}) - \mathsf{F_2}(\relx{x}) \le   \mathsf{F_1}(\relx{x}) - \gamma^\ast\relx{x} $, so $(\relx{x},\relx{t}) \in \inter(\cC_{\relx{x}})$.
\end{proof}

In \cite{munoz2022maximal,serrano2019,xusignomial}, the authors study  the sub/superlevel sets of some DC functions.
Their construction of $\cS$-free sets relies on a common reverse-linearization technique: reverse the set $\cS$ by changing the sign of its defining inequality, and linearize one  convex function.

In  our case, $f$ is an SS function, so we first need to extend the submodular and supermodular components of $f$. After the extension, we obtain a DC function. Then, we can  apply the reverse-linearization technique to its continuous extension.

\section{Two applications}
\label{sec.app}
In this section, we discuss applications of intersection cuts to Boolean multilinear programming and D-optimal design. We exploit  submodular structures in these two problems.

\subsection{Boolean multilinear constraints}
\label{sec.bmc}
We consider \leo{the} construction of $\cS$-free sets for Boolean multilinear constraints.
Since $x\in\{0,1\}\Leftrightarrow x^2=x$, a  polynomial function defined on binary variables is equivalent to a  multilinear function on binary variables.   A Boolean multilinear function is sometimes called a pseudo Boolean function.

 A similar case is the construction of $\cS$-free sets  for continuous quadratic constraints \cite{munoz2022maximal}. We call this construction the ``continuous approach''. It applies eigenvalue decomposition to factor the symmetric matrices representing quadratic terms in a quadratic constraint. After factoring, the reformulated constraint contains a DC function. This reformulation  is amenable to the reverse-linearization technique, by which one obtains the so-called continuous-quadratic-free sets \cite{munoz2022maximal}. Multilinear terms, however, are represented by tensors. It is doubtful whether this construction can be extended so as to produce DC functions from tensors.
 
Here we consider an alternative discrete approach. It exploits the submodularity and the supermodularity of Boolean multilinear functions. In \cite{billionnet1985maximizing,nemhauser1978analysis}, a class of Boolean multilinear functions is shown to be supermodular. We give a submodular-supermodular decomposition for general Boolean multilinear functions in the following.

\begin{proposition}
\label{lem.supml}
Given a Boolean multilinear function $f: \cB \to \bR$ defined as $f(x) := \sum_{k \in [K]} a_k \prod_{j \in A_k}x_j$  with $K$ multilinear terms, where  $A_k \subseteq [n]$. Let $f = f_1 - f_2$ where $f_1(x):=\sum\limits_{k \in [K]\atop a_k < 0} a_k \prod\limits_{j \in A_k}x_j$ and $f_2(x):=-\sum\limits_{k \in [K]\atop a_k > 0} a_k \prod\limits_{j \in A_k}x_j$.  Then  $f_1,f_2$ are submodular over $\cB$.
\end{proposition}
\begin{proof}
Given a Cartesian product set $D:=\prod_{j \in [n]}D_j$ ($D_j \subseteq \bR$), a function $g: D  \to \bR$ is a generalized supermodular function over $D$, if
for every $x, y\in D$, $g(x) + g(y) \le g\left ( x \lor y \right) + g\left(x \land y\right)$.
Each multilinear term function $\prod_{j \in A_k}x_j$ is a Cobb-Douglas function \cite{topkis2011supermodularity}, which is a  generalized supermodular function over $\bR^n_+$. It is known \cite{topkis2011supermodularity} that, if restricting the domain (e.g. $\bR^n_+$)  to its subdomain (e.g. $\cB$) still yields a Cartesian product set, then the supermodularity is preserved. Moreover, a negative combination of supermodular functions is a submodular function. Therefore, $f_1,f_2$ are submodular functions over $\cB$.
\end{proof}

Since every Boolean multilinear function is an SS function, we can construct $\cS$-free sets for the corresponding superlevel set or hypograph set.

\begin{corollary}
\label{prop.bm}
Given a multilinear function $f: \cB \to \bR:x \to f(x) := \sum_{k \in [K]} a_i \prod_{j \in A_k}x_j$ ($A_k \subseteq [n]$) as in \Cref{lem.supml}, assume that $f = f_1 - f_2$ where $f_1(x):=\sum_{k \in [K]\atop a_k < 0} a_k \prod_{j \in A_k}x_j$ and $f_2(x):=\sum_{k \in [K]\atop a_k > 0} -a_k \prod_{j \in A_k}x_j$.  Let $\cS$, $\overline{\cS}$, and $\cC_{\relx{x}}$ be as \eqref{eq.cs}, \eqref{eq.csb}, \eqref{eq.c}, respectively. Then, the set $\cC_{\relx{x}}$ is an $\cS$-free set. Moreover, if $\relx{x} \notin \overline{\cS}$, then $\cC_{\relx{x}}$ does not contain $\relx{x}$ in its interior.
\end{corollary}
\begin{prf}
By \Cref{lem.supml}, we know that both $f_1$ and $f_2$ are submodular. Hence, the result follows by applying  \Cref{prop.dsfree}.
\end{prf}

Importing the notation in \Cref{lem.supml}, a BMP problem has the following form:
\begin{subequations}
\label{bmp}
\begin{alignat}{2}
	\max &&\quad t \\
&& \quad \sum_{k \in \cK_0} a_{ik} \prod_{j \in A_k}x_j &\ge  t  \label{bmp.obj}\\
  \forall i \in  [m] && \quad \sum_{k \in \cK_i} a_{ik} \prod_{j \in A_k}x_j & \ge 0  \label{bmp.cons}\\
  \forall  j \in [n] && \quad 	x_j  & \in \{0,  1\},
\end{alignat}
\end{subequations}
where $m$ is the number of constraints, $K$ is the number of  distinct multilinear terms in the BMP, $\cK_i \subseteq [K]$ is the index set of multilinear terms in the $i$-th constraint ($0$ for objective). Unconstrained BMP has several synonyms: \pbm  or \textsc{multilinear unconstrained binary optimization (MUBO)}.

To construct $\cS$-free sets for Boolean multilinear constraints in the BMP, we need to write them  as the standard form \eqref{eq.cs}.
For all $i \in [m]$ or $i = 0$, let $$f_i(x) := \sum_{k \in \cK_i} a_{ik} \prod_{j \in A_k} x_j,$$ and write $$f_i(x)= f_{i1}(x) - f_{i2}(x),$$ where $f_{i1}:= \sum_{k \in \cK_i:a_{ik} < 0} a_{ik} \prod_{j \in A_k} x_j$ and $f_{i2}:= -\sum_{k \in \cK_i:a_{ik} > 0} a_{ik} \prod_{j \in A_k} x_j$ are two  submodular functions.  

The objective and constraints of \eqref{bmp} can be represented as $$f_{i1}(x) - f_{i2}(x) \ge \ell_i t$$ (for all $i \in [m]$, $\ell_i = 0$, and $\ell_0$ = 1), which, by \Cref{prop.bm}, is in the standard form.

Separating intersection cuts requires LP relaxations or corner polyhedra.
 One can first lift multilinear terms to obtain an extended formulation:
 \begin{subequations}
\label{eq.bmplift}
\begin{alignat}{2}
	\max && \quad t\\
&& \quad \sum_{k \in \cK_0} a_{0k}  y_k &\ge  t  \label{bmplift.obj}  \\
  \forall i \in  [m] && \quad \sum_{k \in \cK_i} a_{ik}  y_k & \ge 0  \label{bmplift.cons}\\
	\quad  k \in [K] && \quad y_k &=  \prod_{j \in A_k}x_j \label{bmplift.monomial} \\
     \forall  j \in [n] && \quad 	x_j &\in \{0,  1\}
\end{alignat}
\end{subequations}

The standard Boolean linearization technique  \cite{crama1993concave} can reformulate a  multilinear term $\prod_{j \in A_k}x_j $ by  its underestimators and overestimators:
 \begin{subequations}
 \label{eq.stdlinear}
\begin{alignat}{2}
    \forall j \in A_k && \quad y_k  & \le x_j  \label{eq.stdlinear.over}\\
           		 && \quad  y_k  & \ge |A_k| + 1 - \sum_{j \in A_k} x_j \label{eq.stdlinear.under}\\
           		 && \quad y_k & \in [0,1]\\
    \forall j \in A_k && \quad x_j & \in \{0,1\},  		 
\end{alignat}
\end{subequations}
where $|A_k|$ is the cardinality of $A_k$. Then,  by linearizing each nonlinear constraint \eqref{bmplift.monomial}  as linear constraints in \eqref{eq.stdlinear}, one obtains a MILP reformulation of \eqref{eq.bmplift}.

To construct LP relaxations and corner polyhedra,  one can simply drop the integrality constraints $x_j \in \{0,1\}$. The direct LP relaxation for the MILP reformulation is also  an  LP relaxation for the BMP \eqref{eq.bmplift}. This gives us a corner polyhedron in the extended space $(x,y,t)$. The $\cS$-free set lives in a projected space (\ie $(x,t)$-space).  By extracting $(x,t)$ entries of rays of the corner polyhedron, we  project the corner polyhedron into the $(x,t)$-space.

Given a corner polyhedron, it is straightforward to construct intersection cuts for the BMP: we  separate intersection cuts constructed from  the $\cS$-free sets given by \Cref{prop.dsfree}.

We note that Boolean quadratic constraints belong to Boolean multilinear constraints, and continuous  quadratic constraints relax Boolean quadratic constraints. Both the continuous and discrete approaches can construct valid $\cS$-free sets for Boolean quadratic constraints. We remark that  maximal continuous-quadratic-free sets are no longer maximally Boolean-quadratic-free. It is easy to see that the  discrete approach  preserves the term-wise sparsity patterns of the SS functions and requires no factorizations. Therefore,  the  discrete approach is computationally amenable to ill-conditioned or sparse coefficient matrices.

\subsection{D-optimal design}
\label{sec.dopt}
In statistical estimation, optimal designs are a class of experimental designs that are optimal
with respect to some statistical criterion. We derive an extended convex MINLP  formulation for the \bdopt problem. In this formulation, the problem is a cardinality-constrained submodular maximization problem.

Let $\bS^m$ denote the set of $m$-by-$m$ symmetric matrices, and let $\bS^m_{+}$ (resp. $\bS^m_{++}$)  denote the set of $m$-by-$m$ positive semi-definite (resp. positive definite) matrices.
Given a set of  full row-rank matrices $\{M_j \in \bR^{m \times r_k}\}_{j \in  [n]}$, an
optimal design problem usually has the following form:
\begin{subequations}
\label{eq.optimal}
\begin{alignat}{2}
    \max  && \quad \Phi \left(\sum_{j \in  [n]}M_j \t{M_j} x_j \right)\\
      && \quad   \sum_{j \in [n]}x_j = k\\   
  \forall  j \in [n] && \quad     x_j  \in \{0,  1\},
\end{alignat}
\end{subequations}
where $k$ is the size of the design and $\Phi: \bS^m \to \bR$ is the design criterion. The matrix $M(x) := \sum_{j \in  [n]}M_j \t{M_j} x_j$ is called the \textit{information matrix}. For the D-optimal criterion \cite{bouhtou2010submodularity,sagnol2015computing}, $\Psi$ is  the log determinant function  $\ldet$.

People usually study \bdopt, where a statistical prior on the parameters $\{M_i\}_{i \in [n]}$ adds a regularization term $\epsilon I$  into the information matrix $M(x)$. This additional term is also due to the well-posedness: when $x = 0$, $ \ldet (\epsilon I) $ is well defined.
Then, the submodular maximization version of the \bdopt problem  has the following  formulation:
\begin{subequations}
\label{eq.optimal2}
\begin{alignat}{2}
    \max  && \quad \ldet \left(\epsilon I + \sum_{j \in  [n]} M_j \t{M_j} x_j  \right) \\
      && \quad   \sum_{j \in [n]}x_j = k\\   
  \forall  j \in [n] && \quad     x_j  \in \{0,  1\},
\end{alignat}
\end{subequations}

The log determinant function is concave and has a semi-definite programming (SDP) and geometric programming representation \cite{aps2018mosek}.  The scalability 
 of the mixed-integer log determinant formulation above is limited by the current state of SDP solvers. Based on the second order cone representation of the determinant function $\det(M(x))$ \cite{sagnol2015computing}, we give an extended formulation for \eqref{eq.optimal2}:
\begin{subequations}
\label{ref.dopt}
  \begin{alignat}{2}
   	 \max \quad &&  t  \\
  	 \quad && t \le \sum_{i \in [m]} \log(J_{ii})  \label{ref.log}\\
      \quad && \sum_{j \in [n] \cup \{0\}} M_j Z_j  = J \\
  	 \quad &&  J \textup{ is lower triangular} \\
  	 \quad  j \in [n] \cup \{0\} \, j \in [m] &&\norm{Z_j e_i}^2 \le u_{ji} x_j  \label{ref.soc}\\
      \quad i \in [m] &&  \sum_{j \in [n] \cup \{0\}} u_{ji} \le J_{ii}\\
      \quad && \sum_{j \in [n]} x_j = k \\
   		 \quad  && x \in \{1\} \times \cB  \\
       		 \quad &&  J \in \bR^{m \times m}\\
      \quad  j \in [n] && Z_j \in \bR^{r_j \times m}\\
      \quad j \in [n] \cup \{0\} \, i \in [m] &&  u_{ji} \in \bR_+^{r_j \times m},
  \end{alignat}
\end{subequations}
where  $M_0 = \epsilon^{1/2}I$ is an auxiliary matrix.
One can represent this formulation by low-dimensional convex cones \cite{aps2018mosek}, \eg (rotated) second-order cones, and exponential cones. Therefore, this extended formulation is amenable to computation.

\begin{proposition}
\eqref{ref.dopt} is equivalent to \eqref{eq.optimal2}, and the objective of \eqref{ref.dopt} is submodular w.r.t. $x$.
\end{proposition}
\begin{proof}
One can modify the original D-optimal design problem by adding a slack variable $x_0 = 1$.
Applying the logarithmic transformation to results in \cite{sagnol2015computing}, \eqref{ref.dopt} is equivalent to \eqref{eq.optimal2}.  It follows from  \cite{SAGNOL2013258,shamaiah2010greedy} that \eqref{ref.dopt} is submodular w.r.t. $x$.
\end{proof}

A global optimization solver like \texttt{SCIP} can linearize the constraints in the extended formulation \eqref{ref.dopt}, and thus produces an LP relaxation in the extended space. We can  obtain a corner polyhedron as the approach dealing with the BMP. Then, we can construct intersection cuts from  hypograph-free sets.

\section{Separation problem}
\label{sec.sep}
In this section,  we consider the separation problem to generate an intersection cut using an $\cS$-free set. Summarizing the previous sections, the $\cS$-free set is in the  form of $$
    \cC := \{(x,t) \in \bR^n \times \bR: \mathsf{G}(x) \le \ell t\},
$$
where  $\mathsf{G}(x) = \max_{s \in \ext(\ep{g})} sx$ is the extended envelope of some submodular function $g$ over $\cB$ and $ \ell \in \{0, 1\}$. We remark that the extended envelope epigraph $\ee{f}$ in \eqref{eq.ee} is a special case with $ \ell = 1$ and $g=f$; the set $\cC_{\relx{x}}$ in  \eqref{eq.c} is also a special case that $g(x)= f_1(x) - \gamma^\ast x$.

Assume that $z^\ast:=(\relx{x}, \relx{t})$ is a vertex of a corner polyhedron $\cR$, and $z^\ast \in \inter(\cC)$. Recalling the cut coefficient formula in \Cref{sec.prem}, the separation problem is reduced to calculate the step length along each ray $r^j$:
\begin{equation}
\label{eq.supstep}
\eta_j^\ast = \sup_{\eta_j \ge 0}\{\eta_j: z^\ast + \eta_j r^j \in \cC\}.   
\end{equation}
This line search problem asks for the step length to the border of $\cC$ along the ray $r^j$ from the interior point $z^\ast$.   We denote by $r^j_x, r^j_t$ the projection of $r^j$ on $x$- and $t$- spaces. Looking at the function defining $\cC$, the intersection step length $\eta_j^\ast$ is the zero point of the following function:
$$\zeta^j:\bR_+ \to \bR: \eta_j \to \zeta^j(\eta_j) := \ell  (\relx{t} + r^j_t \eta_j) -  \mathsf{G}(\relx{x} + r^j_x  \eta_j ).$$
This function enjoys the following properties.

\begin{proposition}
\label{prop.uni}
$\zeta^j$ is a  concave  piece-wise linear function over  $[0,+\infty]$ with  $\zeta^j(0) > 0$.  If $\eta^\ast_j < \infty$ and there exists an $\eta'_j > 0$ with $\zeta^j(\eta'_j) = 0$, then $\eta'_j = \eta^\ast_j$, \ie the solution $\eta^\ast_j$ must be unique.  For all $s^\ast \in \argmax_{s \in \ext(\ep{g})} s (\relx{x}+\eta_j r^j_x)$, $\ell r^j_t- s^\ast r^j_x$ is a subgradient in $\partial \zeta^j( \eta_j)$. For $\eta_j > \eta^\ast_j$, $\partial \zeta^j( \eta_j) \le \partial \zeta^j( \eta^\ast_j)$.
\end{proposition}
\begin{proof}
Since  the extended envelope $\mathsf{G}$ is the maximum of linear functions, it is convex and piece-wise linear, so  $\zeta^j$ is concave and piece-wise linear. Since $\zeta^j(0) = \ell \relx{t} - \mathsf{G}(\relx{x})$, it follows from the assumption $z^\ast \in \inter(\cC)$ that $  \ell \relx{t} >   \mathsf{G}(\relx{x})$ and thus  $\zeta^j(0) > 0$. Since $\cC$ is closed and convex, $\eta'_j = \eta^\ast_j$ if and only if $z^\ast + \eta'_j r^j  \in \bd(\cC)$. That is $\mathsf{G}(r^j_x \eta_j + \relx{x}) = \mathsf{G}(\relx{x}) + r^j_t \eta'_j $, i.e., $\zeta^j(\eta'_j) = 0$. Since $s^\ast \in \partial {\mathsf{G}}(\relx{x}+  r^j_x \eta_j)$, by the chain rule, $\ell r^j_t - s^\ast r^j_x$ is a subgradient of $\zeta^j$. By the concavity of $\zeta^j$, its subgradients are non-increasing.
\end{proof}

 By \Cref{prop.uni}, the line search problem \eqref{eq.supstep} is reduced  into solving a univariate nonlinear equation:
\begin{equation}
\label{eq.nleq}
	\zeta^j(\eta^j) = 0.
\end{equation}
For each ray $r^j$, solving \eqref{eq.nleq} gives the unique zero point of the  univariate function $\zeta^j$, or certificates that no such point exists.

To solve the univariate nonlinear equation \eqref{eq.nleq}, it is natural to deploy a Newton-like algorithm. Therefore, we need the value and (sub)gradient information of $\zeta^j$. Moreover, the  computation of $\zeta^j$ can be reduced to the  computation of $\mathsf{G}$.  A sorting algorithm can compute the value and subgradients of $\mathsf{G}$ (see \Cref{prop.out}). This means that one can compute $\zeta^j$ in a strongly polynomial time.  

Previous works \cite{ChmielaMunozSerrano2023,xusignomial} use the bisection algorithm, which guarantees finding the zero point within a given tolerance. Our implementation   is similar to the \textit{discrete Newton  algorithm} in \cite{Goemans}, but is combined with the bisection algorithm, so we call our implementation the \textit{hybrid discrete Newton  algorithm}. The bisection algorithm helps find a starting point for the Newton algorithm. Thanks to the piece-wise linearity of the univariate function $\zeta^j$, our algorithm can find an exact zero point in a finite time.

\begin{algorithm}[htbp]
 \textbf{Input:} The univariate function $\zeta^j$, (scalar) starting  point $\Delta > 0$ (default: 0.2), a numeric $\eta_{\infty}$ representing $+\infty$, and the maximum number $I$ of search steps (default: 500)\;
 \textbf{Output:}  $\eta_j > 0$ such that $\zeta^j(\eta_j) = 0$\;
 Let $i = 0$, $\eta_{j}  =\Delta$\;
 \uIf{$\zeta^j(\eta_\infty) > 0$}{
    $\eta_j = \eta_\infty$\Comment*[r]{safeguard}
 }
 \Else{
 \While{$i < I$}
 {
Let $ s^\ast \in \argmax_{s \in \ext(\ep{g})} s (\relx{x}+ r^j_x \eta_j)$\;
Compute a subgradient $\beta =r^j_t - s^\ast r^j_x$\;
\uIf{$\zeta^{j}(\eta_j) = 0$}{
\Break\;
}
\ElseIf{$ \beta < 0$}{
$\eta_j = \eta_j - \frac{\zeta^j(\eta_j)}{\beta}$ \Comment*[r]{Newton step}\label{algo:newton.step}
}
\Else{
$\eta_j = 2 \eta_j$\Comment*[r]{bisection step}\label{algo:bis.step}
}
$i = i +1$\;
 }}
\caption{Hybrid discrete Newton algorithm}
\label{algo.newton}
\end{algorithm}

\begin{proposition}
\label{prop.converg}
The hybrid discrete Newton algorithm terminates in a finite number of steps and finds the zero point $\eta^\ast_j$.
\end{proposition}
\begin{proof}

For all $\eta \in \bR_+$, we assume that  \Cref{algo.newton} chooses and computes a unique subgradient $\beta$ at $\eta_j$,  we denote it  $\nabla \zeta^j(\eta_j)$, and call it algorithmic gradient.
The concavity of $\zeta^j$ implies that its algorithmic gradient is monotone-decreasing w.r.t. $\eta_j$. There is a threshold $\eta'_j \ge 0$ such that, for all $\eta_j \in [0, \eta'_j)$, the algorithmic gradient $\nabla\zeta^j(\eta_j) > 0$; for all $\eta_j \in  [\eta'_j, +\infty]$ (called the Newton step region), the algorithmic gradient $\nabla\zeta^j(\eta_j) \le  0$.

After a finite number of bisection steps (at most $\lceil \log(\eta'_j / \Delta)\rceil$), the algorithm enters the Newton step region $[\eta'_j, +\infty]$, where the algorithmic gradient is always negative.
Then, we prove that the algorithmic gradient $\nabla\zeta^j(\eta_j)$ at step $i$ is different from that at step $i-1$, and the algorithm stays in the  Newton step region. Since $\zeta^j$ is piece-wise linear (the number of its distinct algorithmic gradients is finite),  the algorithm must terminate in a finite number of steps.

If at step $i-1$, $\zeta^j(\eta_j -  \frac{\zeta^j(\eta_j)}{\nabla\zeta^j(\eta_j)}) = 0 $, then the algorithm terminates at this step and finds the zero point. If at step $i-1$, $\zeta^j(\eta_j - \frac{\zeta^j(\eta_j)}{\nabla\zeta^j(\eta_j)}) < 0 $, then we prove that $\nabla \zeta^j(\eta_j - \frac{\zeta^j(\eta_j)}{\nabla\zeta^j(\eta_j)}) \ne \nabla \zeta^j(\eta_j)$ and $\nabla \zeta^j(\eta_j - \frac{\zeta^j(\eta_j)}{\nabla\zeta^j(\eta_j)}) \le 0$.

First, assume, to aim at a contradiction, that $\nabla \zeta^j(\eta_j - \frac{\zeta^j(\eta_j)}{\nabla\zeta^j(\eta_j)}) = \nabla \zeta^j(\eta_j)$.  Knowing that the algorithmic gradient is  monotone-decreasing, the piece-wise linearity of  $\zeta^j$ implies that this algorithmic gradient is constant in the range $[\eta_j -  \frac{\zeta^j(\eta_j)}{\nabla\zeta^j(\eta_j)}, \eta_j]$. It follows that for all $\delta \in [0, \frac{\zeta^j(\eta_j)}{\nabla\zeta^j(\eta_j)}]$, $\zeta^j(\eta_j - \delta) = \zeta^j(\eta_j) - \delta   \nabla\zeta^j(\eta_j)$. Hence, $\zeta^j(\eta_j - \frac{\zeta^j(\eta_j)}{\nabla\zeta^j(\eta_j)}) = 0$, which leads to a contradiction.

Second, we show that $\nabla \zeta^j(\eta_j - \frac{\zeta^j(\eta_j)}{\nabla\zeta^j(\eta_j)}) \le 0$. When $\frac{\zeta^j(\eta_j)}{\nabla\zeta^j(\eta_j)} \le 0$, by the mononcity of $\nabla \zeta^j$, $\nabla \zeta^j(\eta_j - \frac{\zeta^j(\eta_j)}{\nabla\zeta^j(\eta_j)}) \le \nabla \zeta^j(\eta_j) < 0$. When $\frac{\zeta^j(\eta_j)}{\nabla\zeta^j(\eta_j)} > 0$, as by assumption that $\nabla\zeta^j(\eta_j) < 0$, $\zeta^j(\eta_j)$ must be negative. Then, by the concavity of $\zeta^j$, $\zeta^j(\eta_j - \frac{\zeta^j(\eta_j)}{\nabla\zeta^j(\eta_j)}) \le \zeta^j(\eta_j) - \nabla\zeta^j(\eta_j)\frac{\zeta^j(\eta_j)}{\nabla\zeta^j(\eta_j)} = 0$. This implies that $\nabla \zeta^j(\eta_j - \frac{\zeta^j(\eta_j)}{\nabla\zeta^j(\eta_j)}) \le 0$.

\end{proof}

From \Cref{prop.converg},  the hybrid discrete Newton algorithm first executes bisection steps with increasing $\eta_j$ and $\zeta^j(\eta_j)$. Then it enters into the Newton step region. After a single Newton step, $\zeta^j(\eta_j)$ becomes negative, and then monotonically
increases to zero in a finite number of steps. The discrete Newton algorithm in \cite{Goemans} is applied to the line search problem for submodular polyhedra, which are polars of extended polymatroids. The algorithm runs in a strongly polynomial time. In our case, $\cC$ includes the extended polymatroid and is unbounded. The corresponding line search problem may have no solutions, and this is a usual case in intersection cut computation \cite{chmiela2022implementation}. Therefore, \Cref{algo.newton} needs a safeguard step, where we evaluate $\zeta^j$ at an user-defined infinity. One may also prove that \Cref{algo.newton} runs in a strongly polynomial time, but a careful analysis for the unbounded case is needed.
Owing to the limitation of pages, we do not expand this topic here.

\section{Computational results}
\label{sec.cresult}
In this section, we conduct computational experiments to test the proposed cuts.

\textbf{Setup and performance metrics.}  The experiments are conducted on a server with  Intel Xeon W-2245 CPU @ 3.90GHz and 126GB main memory. We use \texttt{SCIP} 8.0 \cite{bestuzheva2021scip} as a MINLP framework to solve the natural  formulations of test problems.  \texttt{SCIP} is equipped with \texttt{CPLEX} 22.1 as an LP solver, and \texttt{IPOPT} 3.14 as an NLP solver.

By \Cref{thm.bin}, the simple lifted split $H_j := \{x \in \bR^n: 0 \le x_j \le 1\} \times \bR$ is a maximal $\hyp_{\cB}(f)$-free set,  where the splitting variable $x_j$ is chosen as the most fractional entry of the relaxation solution. In the \textit{standalone} (resp. the \textit{embedded}) configuration, we deactivate (resp. activate)  \texttt{SCIP}'s internal cut separators. Under each configuration, the \textit{submodular cut} (resp. the \textit{split cut}) setting  adds intersection cuts derived from  $\ee{f}$ (resp. $H_j$), and the \textit{default} setting does not add any intersection cuts.

 We focus on  the root node performance and measure the \textit{closed root gap}. Let $d_1$ be the value of the first LP relaxation (without cuts added), let $d_2$ be the dual
bound  after all the cuts are added, and  let $p$ be a reference primal bound. The closed root gap $(d_2-d_1) /(p-d_1)$
 is the closed gap improvement of $d_2$ with
respect to $d_1$.  We also record the number of added cuts, the relative improvement to the default setting, and the total running time. For each configuration and setting,  we compute these statistics' shifted geometric mean (shift value: 1) within a test problem benchmark.

\textbf{Experiment 1:} \maxcut. Consider an undirected graph $G=(V,E,w)$, where $V$ is the set of nodes, $E$ is the set of edges, and $w$ is a weight function over $E$. For a subset $S$ of $V$, its associated cut capacity  is  the sum of the weights of edges with one adjacent node in $S$ and the other in $V \smallsetminus S$. The \maxcut problem aims at finding a subset $S \subseteq V$ with  maximum cut capacity. Using a  binary variable vector $x \in \cB$ indicating whether vertices belong to $S$,  then the problem can be formulated as the following   quadratic unconstrained binary optimization (QUBO) problem: $$
    \max_{x \in \cB}  \sum_{\{i, j\} \in E}w_{ij} ((1-x_i)x_j + x_i(1-x_j)).$$
 When $w$ is nonnegative, the cut capacity function (the objective function) is  submodular. Our benchmark contains 30 ``g05'' and 30 ``pw'' instances with nonnegative weights from Biq Mac  \cite{wiegele2007biq}. The reference primal bounds are also from Biq Mac. The number of vertices is up to 100, and the number of edges is up to 4455. We encode the hypograph reformulation \eqref{eq.milp} of the QUBO. \texttt{SCIP} will automatically reformulate the problem into a MILP via the reformulation-linearization technique  (RLT) \cite{adams1986tight}. This MILP formulation is a special case of the extended formulation \eqref{ref.dopt} of a degree-2 BMP with $m = 0$.

 In \Cref{alpha}, we report the computational results, where ``closed'' denotes the average closed root gap, and ``relative'' denotes its relative value to the default setting. For the standalone (resp. the embedded) configuration, the relative improvement of submodular cuts is $340\%$ (resp. $22\%$) compared to $193\%$  (resp. $21\%$) of split cuts. In the standalone configuration, we can compare the ``clean'' strengths of intersection cuts derived from different hypograph-free sets. Although split cuts are derived from maximal hypograph-free sets and submodular cuts are derived from non-maximal ones, the performance of  split cuts is worse.
 
 We observe that fewer split cuts are generated than submodular cuts. This means that the efficiency of some split cuts does not satisfy \texttt{SCIP}'s internal criteria, so \texttt{SCIP} abandons more split cuts  than submodular cuts. As two types of cuts are derived using the same principle but from different hypograph-free sets,  the distances between the relaxation points to the boundary of hypograph-free sets determine the cut efficiency. This observation suggests that relaxation points are further to the boundary of the extended envelope epigraph than to the splits. Under the embedded configuration, the difference in relative improvements between the two types of cuts is $1\%$, so they perform almost equally. However, the separation time of split cuts is shorter than that of submodular cuts. This is because separating submodular cuts requires solving  nonlinear equations, while the split cuts can be computed  in a closed form.

\begin{table} [htbp]
\centering
\scalebox{0.99}{
\begin{tabular}{c|cr|cccr|cccr}
\toprule
\multirow{2}{*}{\texttt{Configuration}} &
\multicolumn{2}{c|}{\texttt{Default}}    &
\multicolumn{4}{c|}{\texttt{Submodular cut}} &
\multicolumn{4}{c}{\texttt{Split cut}} 	 \\
 & closed & time & closed & relative & time & cuts & closed & relative & time & cuts\\
\hline
standalone & 0.04 & 5.13 & 0.16 & 4.40 & 85.40 & 207.59 & 0.12  & 2.93 & 17.92 & 92.53\\
embedded& 0.22 & 12.62 & 0.27 & 1.22 & 104.02 & 70.68 & 0.27  & 1.21 & 34.62 & 45.15\\
\bottomrule
\end{tabular}}
\caption{Summary of \maxcut results}\label{alpha}
\end{table}

\textbf{Experiment 2:} \pbm.
As  mentioned, \pbm is a MUBO problem, a generalization of QUBO. We can use techniques from \Cref{sec.ss} to generate intersection cuts.

 Our benchmark contains 44 highly dense ``autocorr\_bern'' MUBO instances from MINLPLib \cite{minlplib,Vigerske2022Feb}.
These instances arise in theoretical physics, and the problem is to minimize a degree-four polynomial energy function. The problem is a degree-4 BMP with $m = 0$. \texttt{SCIP}  constructs the extended formulation \eqref{ref.dopt}. The benchmark contains instances with up to 60 binary variables and 3540  Boolean multilinear terms. We use the best-known primal bound from MINLPLib as the reference primal bound.  

In \Cref{beta}, we report the computational results. For the standalone (resp. the embedded) configuration, the relative improvement of submodular cuts is $381\%$ (resp. $13\%$) compared to $131\%$  (resp. $1\%$) of split cuts. In both configurations, the submodular cuts are better than the split cuts in terms of the closed root gap. Moreover, under the embedded configuration, the difference in the relative improvements between the two types of cuts is  around $10\%$. This is larger than $1\%$ of \maxcut benchmark under the same configuration. This divergence between degree-2 and degree-4 MUBO suggests that the submodular cuts are suitable for high-degree Boolean multilinear constraints.

We recall that to solve the nonlinear equations, the hybrid discrete Newton algorithm needs oracle access to the value of the Boolean multilinear function. For some instances, a Boolean multilinear function may consist of thousands of multilinear terms. After  a code timing analysis, we find that the  separation of submodular cuts spends the most  time computing the function value. Therefore, this is the main time performance bottleneck, which needs to be  optimized in the future. An counterintuitive finding is that  non-maximal $\cS$-free sets may yield stronger cuts. This because the geometrical relation between the $\cS$-free sets and corner polyhedron matters.
 
\begin{table} [htbp]
\centering
\scalebox{0.99}{
\begin{tabular}{c|cr|cccr|cccr}
\toprule
\multirow{2}{*}{\texttt{Configuration}} &
\multicolumn{2}{c|}{\texttt{Default}}    &
\multicolumn{4}{c|}{\texttt{Submodular cut}} &
\multicolumn{4}{c}{\texttt{Split cut}} 	 \\
 & closed & time & closed & relative & time & cuts & closed & relative & time & cuts\\
\hline
standalone & 0.01 & 9.49 & 0.05 & 4.81&  43.54  & 43.17  & 0.03 & 2.31 & 14.64 & 20.94\\
embedded& 0.105 & 22.52 & 0.11 & 1.13 & 49.61 & 13.80 & 0.106  & 1.01 & 25.58 & 28.21\\
\bottomrule
\end{tabular}}
\caption{Summary of  \pbm}\label{beta}
\end{table}

\textbf{Experiment 3:} \bdopt.
As mentioned, the \bdopt problem has a submodular maximization form \eqref{eq.optimal2}. In particular, we can encode it as an extended formulation \eqref{ref.dopt} in \texttt{SCIP}. \texttt{SCIP} generates gradient cuts for this convex MINLP. Therefore, we can obtain LP relaxations and corner polyhedra.

Our benchmark consists of two classes of instances. We let  parameters $M_j \in \bR^{m \times 1}$ be single-column matrices. The first class of instances are block design problems \cite{sagnol2015computing}, where $M_j$ are sparse 0-1 matrices. The exact designs correspond to the graphs with a given number of edges and nodes that have a maximum number of spanning trees. Recall that $n$ is the variable dimension, $m$ is the matrix dimension, and $k$ is the cardinality. We generate 15 block design instances with $(n,m,k) \in \{(45,10,9),(55,11,10),(66,12,11)\}$.
The  second class of instances are random Gaussian instances, where $M_j$ are dense real matrices. The entries of  matrices $M_j$ are drawn from a Gaussian distribution with zero mean and $1/\sqrt{n}$ variance. We generate 30 random Gaussian instances with $(n,m) \in \{(50,20),(50,30),(60,24),(60,36),(70,28),(70,42)\}$ and $k \in \{m,m+1,m+2,m+3,m+4\}$. We set the regularization constant $\epsilon $ to $  1e-6$. We use the best primal bound from all settings as the reference primal bound.  Since \texttt{SCIP}'s internal gradient cuts are important for linearizing convex nonlinear constraints, we  keep the gradient cuts but disable all integer-oriented cuts (GMI cuts and mixed-integer rounding cuts etc.) in the standalone configuration.

In \Cref{gamma}, we report the computational results. We divide the results of block design and Gaussian random instances, since the density of  matrices are different. Looking at the default setting in different benchmarks, there is  no difference between the standalone and embedded configurations in terms of the closed root gap. This means that integer-oriented cuts do not improve the root node LP relaxations. We see the same problem for intersection cuts, which do not close the root gap but increase the computing time.  In particular, the number of separated cuts is around one. Thereby, many intersection cuts are too weak to add in the cut pool.

We recall that  intersection cuts and many integer-oriented cuts are LP-based cuts, \ie derived from an LP relaxation of  the extended formulation \eqref{ref.dopt}. Therefore, their strengths depend on the LP relaxation.  A flat corner polyhedron, which usually arises from an LP relaxation with many constraints, may yield weak intersection cuts. Based on types of MINLPs, there are two basic ways to construct initial LP relaxations.  For nonconvex MINLPs, one way usually uses the factorable programming and term-wise envelopes \cite{mccormick1976computability}. Notable examples are Boolean multilinear constraints and  continuous quadratic constraints \cite{munoz2022maximal}.  The McCormick envelopes or Boolean linearization techniques are used to construct their LP relaxations, which have a finite number of constraints.  For convex MINLPs, the other way linearizes nonlinear constraints, and the number of constraints in the LP relaxation can grow to infinite. This is because  a convex nonlinear constraint is equivalent to an infinite number of linear constraints. Since \texttt{SCIP} may add many gradient cuts for approximating the convex MINLP \eqref{ref.dopt}, this yields  flat corner polyhedrons and weak intersection cuts. In summary, the weakness of intersection cuts is due to the flatness of the corner polyhedron.

\begin{table} [htbp]
\centering
\scalebox{0.85}{
\begin{tabular}{c|c|cr|cccr|cccr}
\toprule
\multirow{2}{*}{\texttt{Benchmark}} & \multirow{2}{*}{\texttt{Configuration}} &
\multicolumn{2}{c|}{\texttt{Default}}    &
\multicolumn{4}{c|}{\texttt{Submodular cut}} &
\multicolumn{4}{c}{\texttt{Split cut}} 	 \\
 & & closed & time & closed & relative & time & cuts & closed & relative & time & cuts\\
\hline
 \multirow{2}{*}{\texttt{Block design}} & standalone &  0.59 & 20.46 & 0.59 & 1.0 & 18.71 & 1.84 & 0.59 & 1.0 & 11.62 & 1.77  \\
  &  embedded & 0.59 & 21.44 & 0.59 & 1.0 & 19.0 & 1.84 & 0.59 & 1.0 & 12.41 & 1.77 \\
  \hline
 \multirow{2}{*}{\texttt{Gaussian}}  &  standalone  &  0.83 & 213.13 & 0.83 & 1.0 & 415.07 & 1.45 & 0.83 & 1.0 & 214.17 & 1.45 \\
& embedded  &  0.83 & 214.77 & 0.83 & 1.0 & 426.33 & 1.45 & 0.83 & 1.0 & 214.14 & 1.45 \\
\hline
 \multirow{2}{*}{\texttt{All}}  &  standalone  &   0.75 & 98.47 & 0.75 & 1.0 & 149.54 & 1.57 & 0.75 & 1.0 & 82.6 & 1.55 \\
& embedded  &   0.75 & 100.47 & 0.75 & 1.0 & 153.01 & 1.57 & 0.75 & 1.0 & 84.31 & 1.55 \\
\bottomrule
\end{tabular}}
\caption{Summary of  \bdopt results}\label{gamma}
\end{table}

\section{Conclusion}\label{sec13}

We construct hypograph-free sets for  submodular functions. Our construction relies on a new continuous extension of submodular functions. We characterize maximal hypograph-free sets, \and generalize our results to sets involving submodular-supermodular functions. These yield intersection cuts for Boolean multilinear constraints. We exploit the submodular structure in an extended formulation of the \dopt problem. We propose a hybrid discrete Newton algorithm that can compute intersection cuts efficiently and exactly. The computational results show that intersection cuts derived from the submodularity are stronger than those derived from  split cuts for \maxcut and \pbm problems. For convex MINLPs, our computational results on the \bdopt problem suggest that  corner polyhedra can be flat, which makes intersection cuts weak.


\section*{Statements and Declarations}
Non  conflicts of interest with the journal or the funding agencies.


\bibliographystyle{plain}

\begin{CJK*}{UTF8}{bsmi}
\bibliography{reference}

\begin{thebibliography}{10}

\bibitem{adams1986tight}
Warren~P Adams and Hanif~D Sherali.
\newblock A tight linearization and an algorithm for zero-one quadratic
  programming problems.
\newblock {\em Management Science}, 32(10):1274--1290, 1986.

\bibitem{ahmed2011maximizing}
Shabbir Ahmed and Alper Atamt{\"u}rk.
\newblock Maximizing a class of submodular utility functions.
\newblock {\em Mathematical programming}, 128(1):149--169, 2011.

\bibitem{andersen2010}
Kent Andersen, Quentin Louveaux, and Robert Weismantel.
\newblock {An analysis of mixed integer linear sets based on lattice point free
  convex sets}.
\newblock {\em Mathematics of Operations Research}, 35(1):233--256, feb 2010.

\bibitem{andersen2007}
Kent Andersen, Quentin Louveaux, Robert Weismantel, and Laurence~A. Wolsey.
\newblock Inequalities from two rows of a simplex tableau.
\newblock In Matteo Fischetti and David~P. Williamson, editors, {\em Integer
  Programming and Combinatorial Optimization}, pages 1--15, Berlin, Heidelberg,
  2007. Springer Berlin Heidelberg.

\bibitem{aps2018mosek}
Mosek ApS.
\newblock Mosek modeling cookbook, 2020.

\bibitem{atamturk2020submodularity}
Alper Atamt{\"u}rk and Andr{\'e}s G{\'o}mez.
\newblock Submodularity in conic quadratic mixed 0--1 optimization.
\newblock {\em Operations Research}, 68(2):609--630, 2020.

\bibitem{atamturk2022supermodularity2}
Alper Atamt{\"u}rk and Andr{\'e}s G{\'o}mez.
\newblock Supermodularity and valid inequalities for quadratic optimization
  with indicators.
\newblock {\em Mathematical Programming}, pages 1--44, 2022.

\bibitem{Atamturk2021}
Alper Atamt{\"{u}}rk and Vishnu Narayanan.
\newblock {Submodular function minimization and polarity}.
\newblock {\em Mathematical Programming}, 2021.

\bibitem{balas1971intersection}
Egon Balas.
\newblock Intersection cuts—a new type of cutting planes for integer
  programming.
\newblock {\em Operations Research}, 19(1):19--39, 1971.

\bibitem{basu2010}
Amitabh Basu, Michele Conforti, G{\'{e}}rard Cornu{\'{e}}jols, and Giacomo
  Zambelli.
\newblock {Maximal lattice-free convex sets in linear subspaces}.
\newblock {\em Mathematics of Operations Research}, 35(3):704--720, 2010.

\bibitem{basu2010maximal}
Amitabh Basu, Michele Conforti, G{\'e}rard Cornu{\'e}jols, and Giacomo
  Zambelli.
\newblock Maximal lattice-free convex sets in linear subspaces.
\newblock {\em Mathematics of Operations Research}, 35(3):704--720, 2010.

\bibitem{basu2019}
Amitabh Basu, Santanu~S. Dey, and Joseph Paat.
\newblock {Nonunique lifting of integer variables in minimal inequalities}.
\newblock {\em SIAM Journal on Discrete Mathematics}, 33(2):755--783, 2019.

\bibitem{belotti2015conic}
Pietro Belotti, Julio~C G{\'o}ez, Imre P{\'o}lik, Ted~K Ralphs, and Tam{\'a}s
  Terlaky.
\newblock A conic representation of the convex hull of disjunctive sets and
  conic cuts for integer second order cone optimization.
\newblock In {\em Numerical Analysis and Optimization}, pages 1--35. Springer,
  2015.

\bibitem{bestuzheva2021scip}
Ksenia Bestuzheva, Mathieu Besan{\c{c}}on, Wei-Kun Chen, Antonia Chmiela, Tim
  Donkiewicz, Jasper van Doornmalen, Leon Eifler, Oliver Gaul, Gerald Gamrath,
  Ambros Gleixner, et~al.
\newblock The scip optimization suite 8.0.
\newblock {\em arXiv preprint arXiv:2112.08872}, 2021.

\bibitem{bestuzheva2023global}
Ksenia Bestuzheva, Antonia Chmiela, Benjamin M{\"u}ller, Felipe Serrano, Stefan
  Vigerske, and Fabian Wegscheider.
\newblock Global optimization of mixed-integer nonlinear programs with scip 8.
\newblock {\em arXiv preprint arXiv:2301.00587}, 2023.

\bibitem{bienstock2020outer}
Daniel Bienstock, Chen Chen, and Gonzalo Munoz.
\newblock Outer-product-free sets for polynomial optimization and oracle-based
  cuts.
\newblock {\em Mathematical Programming}, 183(1):105--148, 2020.

\bibitem{billionnet1985maximizing}
Alain Billionnet and Michel Minoux.
\newblock Maximizing a supermodular pseudoboolean function: A polynomial
  algorithm for supermodular cubic functions.
\newblock {\em Discrete Applied Mathematics}, 12(1):1--11, 1985.

\bibitem{bouhtou2010submodularity}
Mustapha Bouhtou, Stephane Gaubert, and Guillaume Sagnol.
\newblock Submodularity and randomized rounding techniques for optimal
  experimental design.
\newblock {\em Electronic Notes in Discrete Mathematics}, 36:679--686, 2010.

\bibitem{minlplib}
Michael~R Bussieck, Arne~Stolbjerg Drud, and Alexander Meeraus.
\newblock Minlplib—a collection of test models for mixed-integer nonlinear
  programming.
\newblock {\em INFORMS Journal on Computing}, 15(1):114--119, 2003.

\bibitem{chen2023multilinear}
Rui Chen, Sanjeeb Dash, and Oktay G{\"u}nl{\"u}k.
\newblock Multilinear sets with two monomials and cardinality constraints.
\newblock {\em Discrete Applied Mathematics}, 324:67--79, 2023.

\bibitem{chmiela2022implementation}
Antonia Chmiela, Gonzalo Mu{\~n}oz, and Felipe Serrano.
\newblock On the implementation and strengthening of intersection cuts for
  qcqps.
\newblock {\em Mathematical Programming}, pages 1--38, 2022.

\bibitem{ChmielaMunozSerrano2023}
Antonia Chmiela, Gonzalo Mu{\~n}oz, and Felipe Serrano.
\newblock Monoidal strengthening and unique lifting in miqcps.
\newblock In {\em Integer Programming and Combinatorial Optimization: 24th
  International Conference, IPCO 2023}, 2023.
\newblock accepted for publication.

\bibitem{coey2020outer}
Chris Coey, Miles Lubin, and Juan~Pablo Vielma.
\newblock Outer approximation with conic certificates for mixed-integer convex
  problems.
\newblock {\em Mathematical Programming Computation}, 12(2):249--293, 2020.

\bibitem{conforti2014}
Michele Conforti, G{\'e}rard Cornu{\'e}jols, and Giacomo Zambelli.
\newblock {\em Integer programming}.
\newblock Springer International Publishing, Cham, 2014.

\bibitem{conforti2015}
Michele Conforti, Gérard Cornuéjols, Aris Daniilidis, Claude Lemaréchal, and
  Jérôme Malick.
\newblock {Cut-Generating Functions and S-Free Sets}.
\newblock {\em Mathematics of Operations Research}, 40(2):276--391, 2015.

\bibitem{conforti2011}
Michele Conforti, Gérard Cornuéjols, and Giacomo Zambelli.
\newblock Corner polyhedron and intersection cuts.
\newblock {\em Surveys in Operations Research and Management Science},
  16(2):105--120, 2011.

\bibitem{coniglio2022submodular}
Stefano Coniglio, Fabio Furini, and Ivana Ljubi{\'c}.
\newblock Submodular maximization of concave utility functions composed with a
  set-union operator with applications to maximal covering location problems.
\newblock {\em Mathematical Programming}, pages 1--48, 2022.

\bibitem{cornuejols2015sufficiency}
G{\'e}rard Cornu{\'e}jols, Laurence Wolsey, and Sercan Y{\i}ld{\i}z.
\newblock Sufficiency of cut-generating functions.
\newblock {\em Mathematical Programming}, 152(1):643--651, 2015.

\bibitem{crama1993concave}
Yves Crama.
\newblock Concave extensions for nonlinear 0--1 maximization problems.
\newblock {\em Mathematical Programming}, 61(1):53--60, 1993.

\bibitem{del2017polyhedral}
Alberto Del~Pia and Aida Khajavirad.
\newblock A polyhedral study of binary polynomial programs.
\newblock {\em Mathematics of Operations Research}, 42(2):389--410, 2017.

\bibitem{del2018multilinear}
Alberto Del~Pia and Aida Khajavirad.
\newblock The multilinear polytope for acyclic hypergraphs.
\newblock {\em SIAM Journal on Optimization}, 28(2):1049--1076, 2018.

\bibitem{del2020impact}
Alberto Del~Pia, Aida Khajavirad, and Nikolaos~V Sahinidis.
\newblock On the impact of running intersection inequalities for globally
  solving polynomial optimization problems.
\newblock {\em Mathematical programming computation}, 12(2):165--191, 2020.

\bibitem{del2022simple}
Alberto Del~Pia and Matthias Walter.
\newblock Simple odd-cycle inequalities for binary polynomial optimization.
\newblock In {\em International Conference on Integer Programming and
  Combinatorial Optimization}, pages 181--194. Springer, 2022.

\bibitem{del2012relaxations}
Alberto Del~Pia and Robert Weismantel.
\newblock Relaxations of mixed integer sets from lattice-free polyhedra.
\newblock {\em 4OR}, 10(3):221--244, 2012.

\bibitem{dey2008}
Santanu~S. Dey and Laurence~A. Wolsey.
\newblock Lifting integer variables in minimal inequalities corresponding to
  lattice-free triangles.
\newblock In Andrea Lodi, Alessandro Panconesi, and Giovanni Rinaldi, editors,
  {\em Integer Programming and Combinatorial Optimization}, pages 463--475,
  Berlin, Heidelberg, 2008. Springer Berlin Heidelberg.

\bibitem{edmonds2003submodular}
Jack Edmonds.
\newblock Submodular functions, matroids, and certain polyhedra.
\newblock In {\em Combinatorial Optimization—Eureka, You Shrink!}, pages
  11--26. Springer, 2003.

\bibitem{fischetti2018}
Matteo Fischetti, Ivana Ljubi{\'{c}}, Michele Monaci, and Markus Sinnl.
\newblock {On the use of intersection cuts for bilevel optimization}.
\newblock {\em Mathematical Programming}, 172(1):77--103, 2018.

\bibitem{fischetti2020}
Matteo Fischetti and Michele Monaci.
\newblock A branch-and-cut algorithm for mixed-integer bilinear programming.
\newblock {\em European Journal of Operational Research}, 282(2):506--514,
  2020.

\bibitem{fortet1960applications}
Robert Fortet.
\newblock Applications de l’algebre de boole en recherche op{\'e}rationelle.
\newblock {\em Revue Fran{\c{c}}aise de Recherche Op{\'e}rationelle},
  4(14):17--26, 1960.

\bibitem{glover1973}
Fred Glover.
\newblock Convexity cuts and cut search.
\newblock {\em Operations Research}, 21(1):123--134, 1973.

\bibitem{Goemans}
Michel~X. Goemans, Swati Gupta, and Patrick Jaillet.
\newblock Discrete newton's algorithm for parametric submodular function
  minimization.
\newblock In Friedrich Eisenbrand and Jochen Koenemann, editors, {\em Integer
  Programming and Combinatorial Optimization}, pages 212--227, Cham, 2017.
  Springer International Publishing.

\bibitem{gomory1969some}
Ralph~E Gomory.
\newblock Some polyhedra related to combinatorial problems.
\newblock {\em Linear algebra and its applications}, 2(4):451--558, 1969.

\bibitem{gomory1963algorithm}
Ralph~E. Gomory.
\newblock Outline of an algorithm for integer solutions to linear programs and
  an algorithm for the mixed integer problem.
\newblock In Michael J{\"u}nger, Thomas~M. Liebling, Denis Naddef, George~L.
  Nemhauser, William~R. Pulleyblank, Gerhard Reinelt, Giovanni Rinaldi, and
  Laurence~A. Wolsey, editors, {\em 50 Years of Integer Programming 1958-2008:
  From the Early Years to the State-of-the-Art}, pages 77--103. Springer Berlin
  Heidelberg, Berlin, Heidelberg, 2010.

\bibitem{han2022fractional}
Shaoning Han, Andr{\'e}s G{\'o}mez, and Oleg~A Prokopyev.
\newblock Fractional 0--1 programming and submodularity.
\newblock {\em Journal of Global Optimization}, pages 1--17, 2022.

\bibitem{hiriart2004fundamentals}
Jean-Baptiste Hiriart-Urruty and Claude Lemar{\'e}chal.
\newblock {\em Fundamentals of convex analysis}.
\newblock Springer Science \& Business Media, 2004.

\bibitem{horsttuy}
R.~Horst and H.~Tuy.
\newblock {\em Global Optimization: Deterministic Approaches}.
\newblock Springer, Berlin, 1990.

\bibitem{khamisov1999optimization}
Oleg Khamisov.
\newblock On optimization properties of functions, with a concave minorant.
\newblock {\em Journal of Global Optimization}, 14(1):79--101, 1999.

\bibitem{kilincc2021joint}
Fatma K{\i}l{\i}n{\c{c}}-Karzan, Simge K{\"u}{\c{c}}{\"u}kyavuz, and Dabeen
  Lee.
\newblock Joint chance-constrained programs and the intersection of mixing sets
  through a submodularity lens.
\newblock {\em Mathematical Programming}, pages 1--44, 2021.

\bibitem{klnc-karzan2015}
Fatma Kılın{\c{c}}-Karzan and Sercan Yıldız.
\newblock {Two-term disjunctions on the second-order cone}.
\newblock {\em Mathematical Programming}, 154(1-2):463--491, 2015.

\bibitem{kilinc-karzan2016}
Fatma Kılınç-Karzan.
\newblock On minimal valid inequalities for mixed integer conic programs.
\newblock {\em Mathematics of Operations Research}, 41(2):477--510, 2016.

\bibitem{liberti2008spherical}
Leo Liberti.
\newblock Spherical cuts for integer programming problems.
\newblock {\em International Transactions in Operational Research},
  15(3):283--294, 2008.

\bibitem{lovasz1983submodular}
L{\'a}szl{\'o} Lov{\'a}sz.
\newblock Submodular functions and convexity.
\newblock In {\em Mathematical programming the state of the art}, pages
  235--257. Springer, 1983.

\bibitem{mccormick1976computability}
Garth~P McCormick.
\newblock Computability of global solutions to factorable nonconvex programs:
  Part i—convex underestimating problems.
\newblock {\em Mathematical programming}, 10(1):147--175, 1976.

\bibitem{modaresi2015}
Sina Modaresi, Mustafa~R. Kilin{\c{c}}, and Juan~Pablo Vielma.
\newblock {Split cuts and extended formulations for Mixed Integer Conic
  Quadratic Programming}.
\newblock {\em Operations Research Letters}, 43(1):10--15, 2015.

\bibitem{modaresi2016}
Sina Modaresi, Mustafa~R. Kılın{\c{c}}, and Juan~Pablo Vielma.
\newblock {Intersection cuts for nonlinear integer programming: convexification
  techniques for structured sets}.
\newblock {\em Mathematical Programming}, 155(1-2):575--611, 2016.

\bibitem{munoz2022towards}
Gonzalo Mu{\~n}oz, Joseph Paat, and Felipe Serrano.
\newblock Towards a characterization of maximal quadratic-free sets.
\newblock {\em arXiv preprint arXiv:2211.05185}, 2022.

\bibitem{munoz2022maximal}
Gonzalo Mu{\~n}oz and Felipe Serrano.
\newblock Maximal quadratic-free sets.
\newblock {\em Mathematical Programming}, 192(1):229--270, 2022.

\bibitem{murota1998discrete}
Kazuo Murota.
\newblock Discrete convex analysis.
\newblock {\em Mathematical Programming}, 83(1):313--371, 1998.

\bibitem{Nemhauser1988}
George Nemhauser and Laurence Wolsey.
\newblock {Matroid and Submodular Function Optimization}, jun 1988.

\bibitem{nemhauser1978analysis}
George~L Nemhauser, Laurence~A Wolsey, and Marshall~L Fisher.
\newblock An analysis of approximations for maximizing submodular set
  functions—i.
\newblock {\em Mathematical programming}, 14(1):265--294, 1978.

\bibitem{rhys1970selection}
John~MW Rhys.
\newblock A selection problem of shared fixed costs and network flows.
\newblock {\em Management Science}, 17(3):200--207, 1970.

\bibitem{richard2010group}
Jean-Philippe~P Richard and Santanu~S Dey.
\newblock The group-theoretic approach in mixed integer programming.
\newblock In {\em 50 Years of Integer Programming 1958-2008}, pages 727--801.
  Springer, 2010.

\bibitem{SAGNOL2013258}
Guillaume Sagnol.
\newblock Approximation of a maximum-submodular-coverage problem involving
  spectral functions, with application to experimental designs.
\newblock {\em Discrete Applied Mathematics}, 161(1-2):258--276, 2013.

\bibitem{sagnol2015computing}
Guillaume Sagnol and Radoslav Harman.
\newblock Computing exact $ d $-optimal designs by mixed integer second-order
  cone programming.
\newblock {\em The Annals of Statistics}, 43(5):2198--2224, 2015.

\bibitem{saxena2011convex}
Anureet Saxena, Pierre Bonami, and Jon Lee.
\newblock Convex relaxations of non-convex mixed integer quadratically
  constrained programs: projected formulations.
\newblock {\em Mathematical programming}, 130(2):359--413, 2011.

\bibitem{schrijver2003combinatorial}
Alexander Schrijver et~al.
\newblock {\em Combinatorial optimization: polyhedra and efficiency},
  volume~24.
\newblock Springer, 2003.

\bibitem{serrano2019}
Felipe Serrano.
\newblock Intersection cuts for factorable {MINLP}.
\newblock In Andrea Lodi and Viswanath Nagarajan, editors, {\em Integer
  Programming and Combinatorial Optimization}, pages 385--398, Cham, 2019.
  Springer International Publishing.

\bibitem{shamaiah2010greedy}
Manohar Shamaiah, Siddhartha Banerjee, and Haris Vikalo.
\newblock Greedy sensor selection: Leveraging submodularity.
\newblock In {\em 49th IEEE conference on decision and control (CDC)}, pages
  2572--2577. IEEE, 2010.

\bibitem{shi2022sequence}
Xueyu Shi, Oleg~A Prokopyev, and Bo~Zeng.
\newblock Sequence independent lifting for a set of submodular maximization
  problems.
\newblock {\em Mathematical Programming}, pages 1--46, 2022.

\bibitem{tawarmalani2005polyhedral}
Mohit Tawarmalani and Nikolaos~V Sahinidis.
\newblock A polyhedral branch-and-cut approach to global optimization.
\newblock {\em Mathematical programming}, 103(2):225--249, 2005.

\bibitem{topkis2011supermodularity}
Donald~M Topkis.
\newblock {\em Supermodularity and complementarity}.
\newblock Princeton university press, Princeton, 2011.

\bibitem{towle2021intersection}
Eli Towle and James Luedtke.
\newblock Intersection disjunctions for reverse convex sets.
\newblock {\em Mathematics of Operations Research}, 47(1):297--319, 2022.

\bibitem{tuy64}
Hoang Tuy.
\newblock Concave programming under linear constraints.
\newblock {\em Soviet Mathematics}, 5:1437--1440, 1964.

\bibitem{thy1985}
Hoang Tuy.
\newblock Concave programming with linear constraints.
\newblock In {\em Doklady Akademii Nauk}, volume 159, pages 32--35. Russian
  Academy of Sciences, 1964.

\bibitem{Vigerske2022Feb}
Stefan Vigerske.
\newblock {MINLPLib: A Library of Mixed-Integer and Continuous Nonlinear
  Programming Instances}, Feb 2022.
\newblock [Online; accessed 1. Feb. 2022].

\bibitem{wiegele2007biq}
Angelika Wiegele.
\newblock Biq mac library—a collection of max-cut and quadratic 0-1
  programming instances of medium size.
\newblock {\em Preprint}, 51, 2007.

\bibitem{xusignomial}
Liding Xu, Claudia D'Ambrosio, Leo Liberti, and Sonia~Haddad Vanier.
\newblock On cutting planes for extended formulation of signomial programming,
  2022.

\bibitem{yu2023strong}
Qimeng Yu and Simge K{\"u}{\c{c}}{\"u}kyavuz.
\newblock Strong valid inequalities for a class of concave submodular
  minimization problems under cardinality constraints.
\newblock {\em Mathematical Programming}, pages 1--59, 2023.

\end{thebibliography}
\end{CJK*}

\end{document}